\documentclass[a4paper,11pt]{amsart}

\usepackage{amsmath,amssymb,amsfonts,graphicx,color,fancyhdr}
\usepackage{float} 
\usepackage[latin1]{inputenc}

\usepackage{psfrag}

\newtheorem{theorem}{Theorem}[section]
\newtheorem{proposition}[theorem]{Proposition}
\newtheorem{lemma}[theorem]{Lemma}
\newtheorem{definition}[theorem]{Definition}

\newtheorem{remark}[theorem]{Remark}

\usepackage[colorlinks=true, pdfstartview=FitV, linkcolor=blue, citecolor=blue, urlcolor=blue]{hyperref}

\usepackage{psfrag,caption}

\allowdisplaybreaks

\def\be#1 {\begin{equation} \label{#1}}
\newcommand{\ee}{\end{equation}}

\def\sqw{\hbox{\rlap{\leavevmode\raise.3ex\hbox{$\sqcap$}}$%
\sqcup$}}
\def\findem{\ifmmode\sqw\else{\ifhmode\unskip\fi\nobreak\hfil
\penalty50\hskip1em\null\nobreak\hfil\sqw
\parfillskip=0pt\finalhyphendemerits=0\endgraf}\fi}

\newcommand{\R}{{\mathbb {R}}}

\newcommand{\Z}{{\mathbb Z}}
\newcommand{\C}{{\mathbb C}}

\DeclareMathOperator*{\essup}{ess\,sup}

\author{Luz Roncal, Saurabh Shrivastava and Kalachand Shuin}
\address{
}
\email{}
\address[Saurabh Shrivastava and Kalachand Shuin]{
Department of Mathematics\\
Indian Institute Science Education and Research Bhopal\\
Bhopal-462066, India}
\email{\{saurabhk,kalachand16\}@iiserb.ac.in}

\address[Luz Roncal]{BCAM - Basque Center for Applied Mathematics \\
48009 Bilbao, Spain and Ikerbasque, Basque Foundation for Science, 48011 Bilbao, Spain}
\email{lroncal@bcamath.org}


\date{\today}

\begin{document}

\title[Bilinear spherical maximal functions of product type]{Bilinear spherical maximal functions of product type}

\begin{abstract}
In this paper we introduce and study a bilinear spherical maximal function of product type in the spirit of bilinear Calder\'{o}n--Zygmund theory. This operator is different from the bilinear spherical maximal function considered by Geba et al. in \cite{GGIP}.
We deal with lacunary and full versions of this operator, and we prove weighted estimates with respect to bilinear weights. Our approach involves sparse forms following ideas by Lacey \cite{Lacey}, but we also use other techniques to handle the particular triplet $(2,2,1)$.
\end{abstract}
\maketitle

\section{Introduction}
The theory of multilinear operators has been an active area of research for the past two decades in harmonic analysis. It finds its roots in the pioneer work by Coifman and Meyer \cite{CM}, although it was the remarkable proof of the boundedness of the bilinear Hilbert transform by Lacey and Thiele \cite{LT1, LT2} that provided the motivation for the study of multilinear singular integrals. The multilinear Calder\'{o}n-Zygmund operators were systematically treated in ~\cite{GT} and later on, in \cite{Lerner},  Lerner et al. developed an appropriate theory of multilinear maximal functions and multilinear weights. In particular, they established weighted boundedness for multilinear Calder\'{o}n-Zygmund operators. Since then there have been several developments in the weighted theory of multilinear weights, we emphasize the recent works~\cite{LMO, Nie} and references therein. 

For notational convenience we shall restrict ourselves to the bilinear setting in this paper. Given locally integrable functions $f_1$ and $f_2$ defined on $\R^n$, the bilinear maximal function $\mathcal{M}(f_1, f_2)$ is defined by
\begin{equation}
\label{maxib}
\mathcal{M}(f_1, f_2)(x):=\sup_{Q \ni x}\prod_{i=1}^{2} \frac{1}{|Q|} \int_{Q} |f_{i}(y_{i})|  \,d y_{i},
\end{equation}
where the supremum in the above is taken over all cubes $Q$ in $\R^n$ containing the point $x$. The cubes are always assumed to have their sides parallel to coordinate axes.    

Note that the bilinear maximal operator $\mathcal M$ is dominated by the product of the classical Hardy-Littlewood maximal functions in a pointwise manner, i.e., 
$$
\mathcal M(f_1, f_2)\leq~M(f_1) M(f_2),
$$ 
where $M$ denotes the Hardy-Littlewood maximal operator given by  
$$
M(f)(x):=\sup_{Q \ni x}\frac{1}{|Q|} \int_{Q} |f(y)|  \,d y.
$$
Let $1 < p_{1}, p_{2}< \infty$ and $p$ be such that $\frac{1}{p}=\frac{1}{p_{1}}+\frac{1}{p_{2}}$. H\"{o}lder's inequality yields that the operator $\mathcal M$ is bounded from $L^{p_1}(w_1)\times L^{p_2}(w_2)\rightarrow L^{p}(w)$ for all $w_i\in A_{p_i}$, $i=1,2$, and $w=\prod_{j=1}^{2} w_{j}^{p/p_{j}}$. Here $A_p$ denotes the class of Muckenhoupt weights, see Subsection \ref{subsec:weights}.

In~\cite{Lerner}, the authors showed that the bilinear maximal operator $\mathcal M$ is the appropriate analogue of the classical Hardy--Littlewood maximal operator. They introduced a suitable analogue of Muckenhoupt weights in the bilinear setting, the class $A_{\vec{P}}$ (see Subsection \ref{subsec:weights}), and showed that the class $A_{\vec{P}}$ is bigger than the product of corresponding linear $A_p$ classes. The class $A_{\vec{P}}$ characterizes the weighted boundedness of the bilinear maximal operator $\mathcal M$. Moreover, the bilinear Calder\'{o}n-Zygmund operators possess weighted boundedness with respect to bilinear weights in $A_{\vec{P}}$. We refer the reader to~\cite{DLP, LN, Lerner} for more details. 

Later on, first in~\cite{LMO} and then in \cite{LMO1, Nie}, the notion of bilinear (or multilinear) weights was further generalised and extrapolation results were proved, see Subsection \ref{subsec:weights}. 

Motivated from the discussions above, in this paper we introduce a bilinear spherical maximal function of product type in the spirit of Calder\'on--Zygmund theory and investigate its weighted boundedness with respect to the bilinear weights just mentioned.

\subsection{Linear spherical maximal functions and bilinear product-type analogues} 
\label{subsec:spherical}

Let $f:\R^n\rightarrow \C$ be a measurable function. Consider the average of $f$ over the sphere of radius $0<r<\infty$ given by 
$$
\mathcal{A}_{r}f(x)=\int_{\mathbb{S}^{n-1}}f(x-ry)\,d{\sigma_{n-1}}(y),
$$ 
where $d\sigma_{n-1}$ is the normalized rotation invariant surface measure on the sphere $\mathbb{S}^{n-1}:=\{x\in \R^n: \|x\|=1\}$. 
The spherical maximal function was introduced by Stein~\cite{Stein} and is defined as 
$$
M_{\operatorname{full}}(f)(x):=\sup_{r>0}\mathcal{A}_{r}f(x), \quad x\in \R^n.
$$
Stein proved that $M_{\operatorname{full}}$ is bounded in $L^p(\R^n)$ if and only if $\frac{n}{n-1}<p\leq \infty$ for all $n\geq 3.$ The problem in dimension $n=2$ was settled later by Bourgain~\cite{Bourgain} (we refer to~\cite{MSS} for a different proof of Bourgain's result).

The dyadic or lacunary version of the spherical maximal function results by taking the supremum over the set $\{2^j:j\in \Z\}$, i.e., 
$$
M_{\operatorname{lac}}(f)(x)=\sup_{j\in \Z}\mathcal{A}_{2^j}f(x).
$$
The lacunary spherical maximal operator $M_{\operatorname{lac}}$ is bounded in $L^p(\R^n)$ for all $1<p\leq \infty$ and $n\geq 2$, see~\cite{Calderon2,CW} for details. Weighted boundedness properties of the spherical maximal operators have been studied in~\cite{CCG, DMO, DV,Manna}. 

In a recent article, Lacey \cite{Lacey} revisited the spherical maximal function and, using a new approach that unified the lacunary and full versions, he  managed to prove sparse bounds for these operators which led him to obtain new weighted norm inequalities. We also refer to~\cite{Lacey} for a discussion about the suitability of $A_p$ weights in the context of the spherical maximal function.

In this paper we introduce a bilinear analogue of the spherical maximal function in the spirit of the bilinear Hardy-Littlewood maximal function \eqref{maxib}, which plays a key role in the theory of bilinear Calder\'{o}n--Zygmund operators. Define 
\begin{equation*}
\mathcal{M}_{\operatorname{full}}(f_{1},f_{2})(x):=\sup_{t>0}\mathcal{A}_{t}f_{1}\mathcal{A}_{t}f_{2}(x).
\end{equation*} 
As earlier, if we take the supremum in the above over the dyadic numbers, we get bilinear analogue of the lacunary spherical maximal function. This way, the bilinear lacunary spherical maximal operator $\mathcal{M}_{\operatorname{lac}}$ is defined as 
$$
\mathcal{M}_{\operatorname{lac}}(f_{1},f_{2})(x):=\sup_{j\in \mathbb{Z}}\mathcal{A}_{2^{j}}f_{1}\mathcal{A}_{2^{j}}f_{2}(x).
$$
We refer to these operators as \textit{bilinear spherical maximal functions of product type}.

Note that $\mathcal{M}_{\operatorname{full}}(f_{1},f_{2})$ (and $\mathcal{M}_{\operatorname{lac}}(f_{1},f_{2})$) is dominated by the product of the linear full (respectively lacunary) spherical maximal functions in a pointwise sense. 
Therefore, H\"{o}lder's inequality immediately yields the $L^{p_1}\times L^{p_2}\rightarrow L^p$ estimates for the operators $\mathcal{M}_{\operatorname{full}}$ and $\mathcal{M}_{\operatorname{lac}}$. In fact, we also get the weighted estimates for the operator with respect to product weights, see Theorem \ref{product}. We will prove new weighted estimates for the bilinear spherical maximal functions with respect to bilinear weights that are beyond the type of weights as described in Theorem~\ref{product}. This result is stated in Theorem~\ref{lacweighted}: We exploit the ideas from~\cite{Lacey} and establish a sparse domination principle for the bilinear spherical maximal functions in Theorem \ref{mainthm1:lac} so that we deduce weighted estimates as a consequence of known results in the literature. On the other hand, in Theorem  \ref{radial} we will provide weighted estimates for the triplet $(2,2,1)$ that cannot be deduced from the sparse domination.

A different analogue of the spherical maximal function in the bilinear setting has been studied in the literature. It was introduced in~\cite{GGIP} and is defined as follows:
\begin{equation}
\label{MGeba}
\mathcal{M}_{\operatorname{sph}}(f_{1},f_{2})(x):=\sup_{t>0}\int_{\mathbb S^{2n-1}}|f_{1}(x-ty)f_{2}(x-tz)|d\sigma_{2n-1}(y,z).
\end{equation}  
In ~\cite{BGHH, GHH} the authors proved partial results obtaining $L^{p_1}\times L^{p_2}\rightarrow L^p$ estimates for the operator $\mathcal{M}_{\operatorname{sph}}$ for a certain range of $p_1, p_2$ and $p$ and some assumptions on the dimension $n$. In~\cite{JL} the authors proved the following pointwise domination result
\begin{equation}
\label{jeole}
\mathcal{M}_{\operatorname{sph}}(f_{1},f_{2})(x)\lesssim M_{\operatorname{full}} (f_1)(x) M(f_2)(x),
\end{equation}
and extended the $L^{p_1}\times L^{p_2}\rightarrow L^p$ estimates for the operator $\mathcal{M}_{\operatorname{sph}}$ to the best possible range of exponents $p_1, p_2$ and $p$ for all $n\geq 2$ (note that an estimate similar to \eqref{jeole} holds with the roles of $M_{\operatorname{full}}$ and $M$ interchanged due to symmetry). 
We also refer to the recent papers~\cite{AP, Dosidis} for the generalisation of the bilinear spherical maximal function to the multilinear setting. Weighted estimates for the bilinear maximal operator $\mathcal{M}_{\operatorname{sph}}$ defined in \eqref{MGeba} beyond the ones that can be obtained trivially from the pointwise estimate \eqref{jeole} remain as an open problem.

The paper is organised as follows. We state the main results in the next section, then in Section~\ref{notdef} we recall necessary definitions and results and also set notation that we use in the paper. Section~\ref{proofwe} is devoted to prove weighted estimates for the operators under consideration and we complete the proofs of Theorems~\ref{lacweighted} and \ref{radial} in this section.  In Section~\ref{compar} we discuss some examples comparing the weighted results obtained in Theorem~\ref{lacweighted} with the H\"{o}lder type results. Next, in Section~\ref{sec:proof} we give the proof of sparse domination result Theorem~\ref{mainthm1:lac}. Finally, in Section \ref{neces} we provide the necessity of some conditions for such a sparse domination.
\section{Main results}
\label{subsec:sparse}
Our first main result is the following theorem containing weighted estimates for the product type operators with bilinear weights in the class defined in Definition \ref{lmodef}. In what follows, we will denote by $L_n$ the triangle with vertexes $(0,1), (1,0)$ and $\big(\frac{n}{n+1}, \frac{n}{n+1}\big)$ and by $F_n$ the trapezium with vertexes $(0,1)$, $\big(\frac{n-1}{n}, \frac{1}{n}\big)$, $\big(\frac{n-1}{n}, \frac{n-1}{n}\big)$ and $\big(\frac{n^2-n}{n^2+1}, \frac{n^2-n+2}{n^2+1}\big)$, see Figure \ref{LnFn}.

\begin{figure}
\includegraphics[scale=0.63]{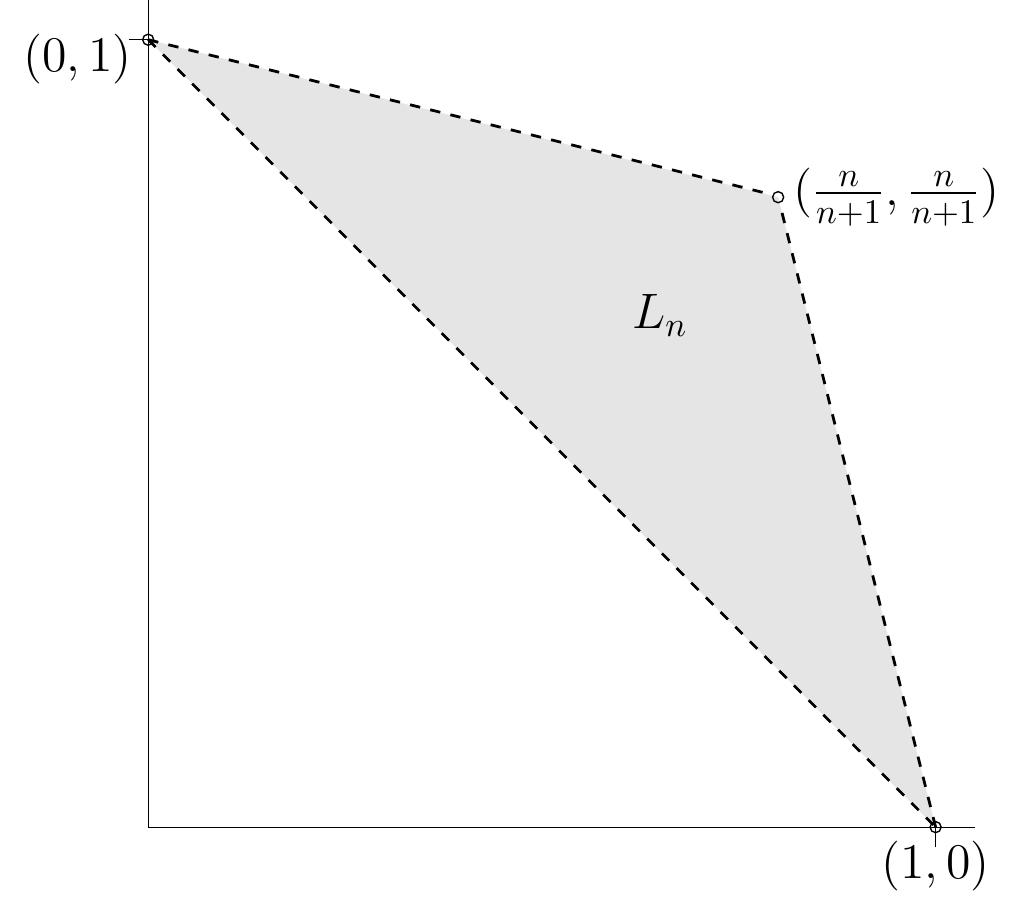}
\includegraphics[scale=0.63]{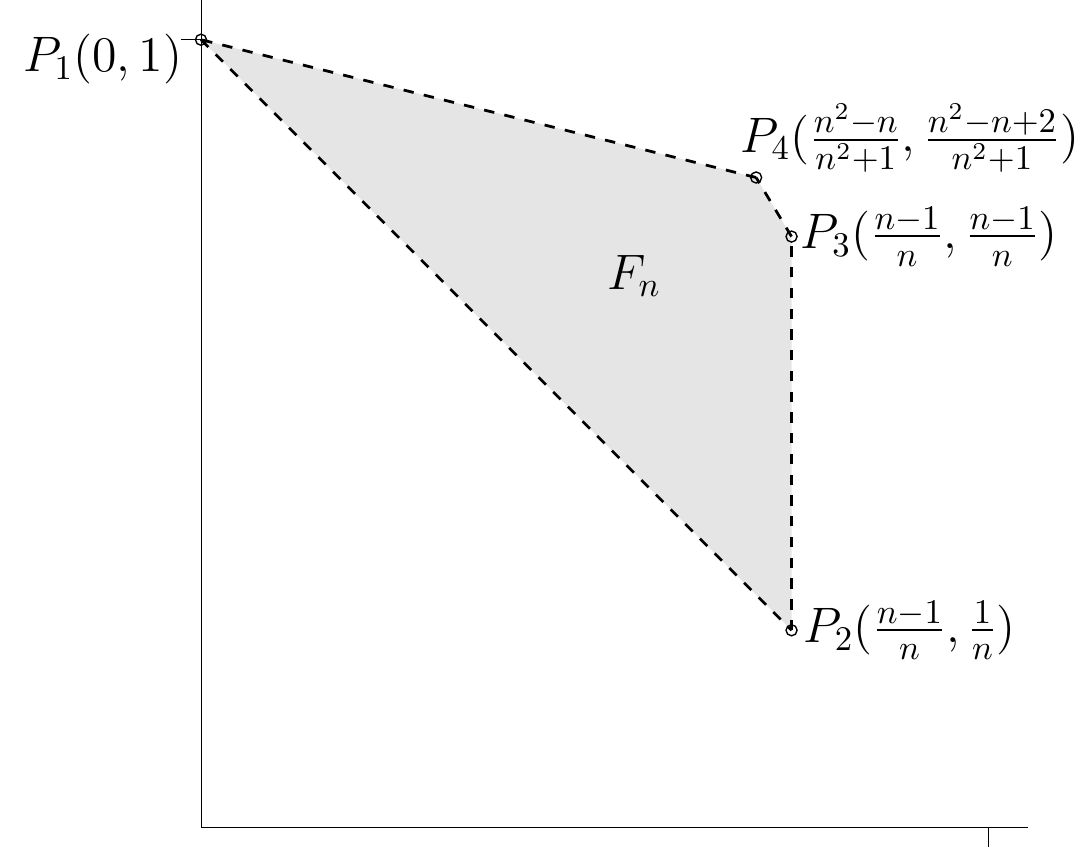}
\caption{Triangle $L_n$ on the left and trapezium $F_n$ on the right.}
\label{LnFn}
\end{figure}

\begin{theorem} 
\label{lacweighted}
Let $n\geq 2$. For  $i=1,2$, let $(\frac{1}{r_{i}},\frac{1}{s_{i}})$ be in the interior of $L_n$ (respectively $F_n$). Assume that
$\frac{1}{r_{1}}+\frac{1}{r_{2}}<1$ and $t=\frac{s_{1}s_{2}}{s_{1}+s_{2}-s_{1}s_{2}}>1$. Then for all $\vec{q}=(q_1,q_2)$, $\frac{1}{q}=\frac{1}{q_1}+\frac{1}{q_2}$ with $r_{i} \leq q_{i}$, $i=1,2$, and $t'>q$,  the operator $\mathcal M_{\operatorname{lac}}$ (respectively $\mathcal M_{\operatorname{full}}$) extends to a bounded operator from $L^{q_1}(w_1)\times L^{q_2}(w_2) \rightarrow L^{q}(w)$, i.e.,  
\begin{equation*}
	\| \mathcal{M}(f_{1},f_{2})\|_{L^{q}(w)} \leq C([\vec{w}]_{A_{\vec{q},\vec{r}}}) \prod^{2}_{i=1}\| f_{i}\Vert_{L^{q_{i}}(w_{i})},
\end{equation*}
where $\mathcal{M}:=\mathcal{M}_{\operatorname{lac}}$ (respectively  $\mathcal{M}_{\operatorname{full}}$) and $\vec{w}=(w_1,w_2)\in A_{\vec{q},\vec{r}}$ with $ \vec{r}=(r_1,r_2,t)$ defined as in Definition \ref{lmodef}.
\end{theorem} 

The weighted estimates in Theorem \ref{lacweighted} are indeed consequence of a sparse domination principle for the bilinear spherical maximal functions shown in Theorem \ref{mainthm1:lac} below. Actually, one could state an improved result, providing the quantitative bounds, including end-points, and vector-valued inequalities, see Theorem \ref{lacweightedNie} and Remark \ref{rem:vv}. For these consequences we appeal to \cite{LMO1, LMO,Nie}. 

Before stating the sparse domination result let us set up the notation. A collection of cubes $\mathcal{S}$ in $\R^n $ is said to be $ \eta$-sparse, $0<\eta<1,$ if there are sets $\{E_S \subset S:S\in \mathcal{S}\}$ which are pairwise disjoint and satisfy $|E_S|>\eta|S|$ for all $S\in \mathcal{S}$. 
By the term  $(p,q,r)$-sparse form we mean the following: 
\[\Lambda_{\mathcal{S}_{p,q,r}}(f,g,h):=\sum_{S\in\mathcal{S}}|S|\langle f\rangle_{S,p}\langle g\rangle_{S,q}\langle h\rangle_{S,r},
\]
see Section \ref{notdef} for notations.

\begin{theorem}
\label{mainthm1:lac}
Let $n\geq 2$.  For $i=1,2$, let $(\frac{1}{r_{i}},\frac{1}{s_{i}})$ be in the interior of $L_n$ (respectively $F_n$). Suppose $\rho_i>r_i$, are such that $\frac{1}{\rho_{1}}+\frac{1}{\rho_{2}}<1$. Then for any non-negative compactly supported bounded functions $f_{1},f_{2}$ and $h$, there exists a sparse collection $\mathcal{S}=\mathcal{S}_{\rho_{1},\rho_{2},t}$ such that 
\begin{equation*}
\langle \mathcal{M}(f_{1},f_{2}),h\rangle\leq C \Lambda_{\mathcal{S}_{\rho_{1},\rho_{2},t}}(f_{1},f_{2},h),
\end{equation*}
where $t=\frac{s_{1}s_{2}}{s_{1}+s_{2}-s_{1}s_{2}}>1$ and $\mathcal{M}:=\mathcal{M}_{\operatorname{lac}}$ (respectively  $\mathcal{M}_{\operatorname{full}}$).
\end{theorem}
We prove this theorem in two steps. First, we shall establish an analogous result, in fact a slightly stronger version of the above theorem, for characteristic functions. Then we shall obtain the theorem for general functions. The proof of these results and of Theorem~\ref{mainthm1:lac} will be given in Section \ref{sec:proof}.

Note that Theorem~\ref{lacweighted} does not provide weighted boundedness of the operators $\mathcal M_{\operatorname{lac}}$ and $\mathcal M_{\operatorname{full}}$ for the tuple $(2,2,1)$. We shall establish some non-trivial weighted estimates for such a tuple. We exploit ideas from  \cite{JSK,XQ} to obtain Theorem \ref{radial}, based on interpolation of analytic families of linear operators in \cite{Calderon}.  
\begin{theorem}
\label{radial}
Let $\mathcal{M}$ be either $\mathcal{M}_{\operatorname{lac}}$ or $\mathcal{M}_{\operatorname{full}}$. Then $\mathcal{M}$ is bounded from $L^{2}(|x|^{\alpha})\times L^{2}(|x|^{\beta})$ to $L^{1}(|x|^{\frac{\alpha+\beta}{2}})$ for $\alpha,\beta$ satisfying:
\begin{itemize}
\item If  $\mathcal{M}=\mathcal{M}_{\operatorname{lac}}$,
	\begin{equation*}
		2(1-n)<\alpha,\beta<n-1\quad \text{and}\quad \alpha+\beta>2(1-n), \quad n\ge2.
	\end{equation*}
\item  If  $\mathcal{M}=\mathcal{M}_{\operatorname{full}}$,
	\begin{equation*}
	2(1-n)<\alpha,\beta<n-2\quad \text{and}\quad \alpha+\beta>2(1-n),\quad n\geq3.
	\end{equation*}
\end{itemize}
\end{theorem}
We would like to remark that while proving Theorem~\ref{radial}, we actually get weighted boundedness of operators $\mathcal{M}_{\operatorname{lac}}$ (and $\mathcal{M}_{\operatorname{full}}$) for the triplet $(2,2,1)$ for more general weights than stated in the theorem above. Moreover,  these weights do not come from the product type bilinear weights. Let $\frac{1}{\phi_{\operatorname{lac}}(\frac{1}{r})}$ denote the piecewise linear function on the interval $(0,1)$ whose graph connects the points $(0,1), (\frac{n}{n+1},\frac{n}{n+1})$ and $(1,0)$, i.e., 
\begin{equation}
\label{philac}
\frac{1}{\phi_{\operatorname{lac}}(\frac{1}{r})}= 
     \begin{cases}
       1-\frac{1}{rn}, &\quad\text{if}\quad 0<\frac{1}{r}\leq \frac{n}{n+1}\\
       n(1-\frac{1}{r}), &\quad \text{if}\quad \frac{n}{n+1}<\frac{1}{r}<1.
     \end{cases}
\end{equation}
Similarly, let $\frac{1}{\phi_{\operatorname{full}}(\frac{1}{r})}$ denote the piecewise linear function on $(0,\frac{n-1}{n})$ whose graph connects the points $(0,1)$, $(\frac{n^{2}-n}{n^{2}+1},\frac{n^{2}-n+2}{n^{2}+1})$ and $(\frac{n-1}{n}, \frac{n-1}{n})$. 
An inspection of the proof of Theorem~\ref{radial} delivers the following (see Section \ref{notdef} for the definitions of weights).  
\begin{proposition}\label{remark} 
We have the following:
\begin{itemize}
\item The operator $\mathcal{M}_{\operatorname{lac}}$ is bounded from $L^{2}(w_{1})\times L^{2}(w_{2})$ to $L^{1}(w)$ for certain weights  $\vec{w}=(w_{1},w_{2})$ which do not belong to product type weights 
 \begin{equation*}
 \label{forbilac}
	\bigcup\limits_{1<r_1,r_2<2}\Big( \prod_{i=1}^2 A_{\frac{2}{r_{i}}}\cap \operatorname{RH}_{\big(\frac{\phi'_{\operatorname{lac}}(\frac{1}{r_{i}})}{2}\big)'}\Big)\bigcup
	 	\left(\mathcal{R}_{2}\times\mathcal{R}_{2}\right),
\end{equation*}
where 
\begin{equation}
\label{Rp}
\mathcal{R}_{p}=\{|x|^{b}:1-n\leq b<(n-1)(p-1)\},\quad  n\ge 2.
\end{equation}
\item The operator $\mathcal{M}_{\operatorname{full}}$ is bounded from $L^{2}(w_{1})\times L^{2}(w_{2})$ to $L^{1}(w)$ for certain weights  $\vec{w}=(w_{1},w_{2})$ which do not belong to product type weights
 \begin{equation*}
 \label{forbifull}
	\bigcup\limits_{\frac{n}{n-1}<r_1,r_2<2}\Big( \prod_{i=1}^2 A_{\frac{2}{r_{i}}}\cap \operatorname{RH}_{\big(\frac{\phi'_{\operatorname{full}}(\frac{1}{r_{i}})}{2}\big)'}\Big)\bigcup
	 	\left(\mathcal{\widetilde{R}}_{2}\times\mathcal{\widetilde{R}}_{2}\right),
\end{equation*}
		where 
\begin{equation}
\label{Rpt}
\mathcal{\widetilde{R}}_{p}=\{|x|^{b}:1-n< b<(n-1)(p-1)-1\},\quad n\ge 3.
\end{equation}
\end{itemize}
\end{proposition}
\begin{remark}
The restriction $n\ge 3$ in Theorem \ref{radial} and in Proposition~\ref{remark} for the case of $\mathcal{M}_{\operatorname{full}}$ arises because the operator $\mathcal{M}_{\operatorname{full}}$ is not bounded from $L^2(\R^2)\times L^2(\R^2)$ into  $L^1(\R^2)$. Indeed, the underlying reason is that in dimension $n=2$,  the linear operator $M_{\operatorname{full}}$ is $L^{p}(\R^n)$ bounded only for $p>2$.
 \end{remark}
\section{Notations and definitions}
\label{notdef}
In this section we collect some of the notations and definitions that we use in this paper. With the letters $c, C\ldots $ we denote structural constants that depend only on the dimension and on parameters. Their values might vary from one occurrence to another, and in most of the cases we will not track the explicit dependence. We will write $\gamma_1\lesssim \gamma_2$ if $\gamma_1\le c \gamma_2$ for a structural constant~$c$. Given $p\ge1$, the conjugate exponent of $p$ will be denoted by $p'$, i.e., $1/p+1/p'=1$.

For any cube $Q$ and $1<p<\infty$, we define
 \[ \langle f\rangle_{Q,p}:=\bigg(\frac{1}{|Q|}\int_Q|f(x)|^pdx\bigg)^{1/p}, \qquad  \langle f\rangle_{Q}:=\frac{1}{|Q|}\int_Q|f(x)|dx,
 \]
where $|Q|$ denotes the Lebesgue measure of $Q$.

A weight is a non-negative locally Lebesgue integrable function that is non-zero in a set of positive measure. We say that a weight $w$ belongs to the Muckenhoupt class $A_p$ if
\begin{equation*}
[w]_{A_p}:= \sup_{Q} \Big(\frac{1}{|Q|}\int_{Q}w\, dx \Big) \Big(\frac{1}{|Q|}\int_{Q} w^{1-p'} \,dx\Big)^{p-1} < \infty, \quad 1<p<\infty.
\end{equation*} 
 The quantity $[w]_{A_p}$ is referred to as the $A_p$ characteristic of $w\in A_p$. 
For $p=1$ the class $A_1$ consists of all $w$ such that 
$$
[w]_{A_1}:=\essup\frac{M (w)}{w}<\infty.
$$

Given $s>1$, a weight belongs to the reverse H\"older $\operatorname{RH}_s$ if there exists a constant $C$ such that, for every cube $Q$ in $\R^n$ with sides parallel to the coordinate axes,
$$
\Big(\frac{1}{|Q|}\int_{Q}w^s\, dx \Big)^{1/s}\le \frac{C}{|Q|}\int_{Q} w \,dx < \infty.
$$

\subsection{Bilinear weights}
\label{subsec:weights}
Let $1\leqslant p_{1}, p_{2}< \infty$ and $p$ be such that 
\begin{equation}
\label{holder}
\frac{1}{p}=\frac{1}{p_{1}}+\frac{1}{p_{2}}.
\end{equation}

\begin{definition}\cite[Definition 3.5]{Lerner}
Let $\vec{p}=(p_{1},p_{2})$. For a given pair of weights $\vec{w}=(w_{1},w_{2})$, set $
w:=\prod_{i=1}^{2} w_{i}^{p/p_{i}}$.
We say that $\vec{w}\in A_{\vec{P}}$ if 
\begin{equation*}
[\vec{w}]_{A_{\vec{P}}}:= \sup_{Q} \Big(\frac{1}{|Q|}\int_{Q}w\, dx \Big) \prod_{j=1}^{2}\Big(\frac{1}{|Q|}\int_{Q} w_{j}^{1-p'_{j}} \,dx\Big) ^{p/{p_{j}^\prime}}  <\infty.
\end{equation*} 
When $p_j=1$, $\big(\frac{1}{|Q|}\int_Qw_j^{1-p'_j}\big)^{1/p'_j}$ is understood as $(\inf_Q w_j)^{-1}$. The quantity $[\vec{w}]_{A_{\vec{P}}}$ is referred to as the bilinear $A_{\vec{P}}$ characteristic of the bilinear weight 
$\vec{w}$. 
\end{definition}
The bilinear $A_{\vec{P}}$ class was further generalised recently in \cite{LMO}. 
\begin{definition}\cite[Section 1]{LMO}
\label{lmodef}
Let $\vec{p}=(p_{1},p_{2})$ and $p$ be as in \eqref{holder}. For a tuple $\vec{r}=(r_{1},r_{2},r_{3})$ with $r_{i}\leq p_{i}$, $i=1,2$, and $r'_{3}>p$, where $1\leq r_{1},r_{2},r_{3}<\infty$, we say that $\vec{w}=~(w_{1},w_{2})\in~A_{\vec{p},\vec{r}}$ if $0<w_{i}<\infty$ a.e. for $i=1,2$ and 
\begin{equation*}
	[\vec{w}]_{A_{\vec{p},\vec{r}}}:=\sup_{Q\subset\mathbb{R}^{n}}\langle w^{\frac{r'_{3}}{r'_{3}-p}} \rangle^{\frac{1}{p}-\frac{1}{r'_{3}}}_{Q}\prod^{2}_{i=1}\langle w^{\frac{r_{i}}{r_{i}-p_{i}}}_{i} \rangle^{\frac{1}{r_{i}}-\frac{1}{p_{i}}}_{Q}<\infty,
	\end{equation*}
where $w:=\prod^{2}_{i=1}w^{p/p_{i}}_{i}$. When $r_{3}=1$, the term corresponding to $w$
	needs to be replaced by $\langle w\rangle^{1/p}_{Q}$. Analogously, when $p_{i}=r_{i}$, the term corresponding to $w_{i}$ needs to be replaced by $\essup_{Q}w^{-1/p_{i}}_{i}$. 
	\end{definition}

\begin{remark}
Note that $A_{\vec{p},(1,1,1)}$ agrees with the class $A_{\vec{P}}$. 
\end{remark}

The following result from~\cite{LMO} describes the bilinear weights $A_{\vec{p},\vec{r}}$ in terms of the classical $A_p$ weights. This provides a useful tool in the study of weighted estimates with respect to bilinear weights. 
\begin{lemma}\cite[Lemma 5.3]{LMO}\label{weightrelation}
Let $\vec{p}=(p_{1},p_{2})$ with $1<p_{1},p_{2},p_{3}<\infty$ and $\vec{r}=(r_{1},r_{2},r_{3})$ with $1\leq r_{1},r_{2},r_{3}<\infty$. Let $p':=p_{3}$ and $\frac{1}{r}:=\sum^{3}_{i=1}\frac{1}{r_{i}}$. Assume that $r_{i}\leq p_{i}$ for $i=1,2$ and $r'_{3}>p$. Consider 
\begin{equation*} 
\frac{1}{\delta_{i}}=\frac{1}{r_{i}}-\frac{1}{p_{i}} \quad 
\text{ and } \quad 
\frac{1}{\theta_{i}}=\frac{1-r}{r}-\frac{1}{\delta_{i}},\quad  i=1,2,3.
\end{equation*}
Then  $\vec{w}=(w_{1},w_{2})\in A_{\vec{p},\vec{r}}$ if and only if  
	\begin{equation*}
	w^{\frac{\theta_{i}}{p_{i}}}_{i}\in A_{\frac{1-r}{r}\theta_{i}}  \quad \text{ with }\quad [w^{\frac{\theta_{i}}{p_{i}}}_{i}]_{A_{\frac{1-r}{r}\theta_{i}}}\leq [\vec{w}]^{\theta_{i}}_{A_{\vec{p},\vec{r}}},\quad i=1,2
	\end{equation*}
	and
	$$
	w^{\frac{\delta_{3}}{p}}\in A_{\frac{1-r}{r}\delta_{3}} \quad \text{ with } \quad [w^{\frac{\delta_{3}}{p}}]_{ A_{\frac{1-r}{r}\delta_{3}}}\leq [\vec{w}]^{\delta_{3}}_{A_{\vec{p},\vec{r}}}.
	$$
\end{lemma}

In \cite{Nie}, Nieraeth presented an alternative approach to describe the bilinear weights $A_{\vec{p},\vec{r}}$ and defined yet another class of weights that is equivalent to the class defined in \cite{LMO}. Nieraeth extended the extrapolation results contained in \cite{LMO} in several directions. 
\begin{definition}\cite[Definition 2.1]{Nie} 
\label{wNier}
Let $\vec{p}=(p_1,p_2)$, $\vec{q}=(q_1,q_2)$ with $p_1,p_2\in (0,\infty)$ and $q_1, q_2\in (0,\infty]$. Let $q$ be given by $\frac{1}{q}=\frac{1}{q_1}+\frac{1}{q_2}$. We say $(\vec{p},s)\leq \vec{q}$ if $\vec{p}\leq \vec{q}$ and $q\leq s$ where $s\in (0,\infty]$. Here $\vec{p}\leq \vec{q}$ means that $p_i\leq q_i$, $i=1,2$. For weights $w_1, w_2$ write $w=\prod_{i=1}^2w_i$. We say that $\vec{w}=(w_1,w_2) \in A_{\vec{q},(\vec{p},s)}$ if 
\begin{equation*}
[\vec{w}]_{\vec{q},(\vec{p},s)}:=\sup_{Q} \Big(\prod_{i=1}^2 \langle w^{-1}_i\rangle_{\frac{1}{\frac{1}{p_i}-\frac{1}{q_i}}, Q} \langle w\rangle_{\frac{1}{\frac{1}{q}-\frac{1}{s}}, Q}\Big) <\infty,
\end{equation*}
where the supremum in the above is taken over all cubes (with sides parallel to coordinate axes) in $\mathbb R^n$. 
\end{definition} 
\begin{remark}
Note that the definition above includes the case $q_j=\infty.$ In this case the norm is interpreted as $\|f_j\|_{L^{q_j}(w^{q_j}_j)}=\|f_jw_j\|_{L^{\infty}}.$ Also, the definition is used with $\frac{1}{q_j}=0$ when $q_j=\infty.$ We refer to \cite{Nie} for more details on this. We would like to refer the reader to \cite{LMO1}, where authors consider a slightly different approach to include the end-points cases which allows one or more indices to take value infinity. Further, note that when $q_j$ are finite, the following relation holds: $(w^{q_{1}}_{1},w^{q_{2}}_{2})\in A_{\vec{q},(r_{1},r_{2},t)}$ if and only if $\vec{w}\in A_{\vec{q},(\vec{r},t')}$.
\end{remark}

\section{Proofs of weighted estimates} 
\label{proofwe}

\subsection{Proof of Theorem \ref{lacweighted}}

As pointed out earlier, the proof of Theorem~\ref{lacweighted} follows from the sparse domination result Theorem~\ref{mainthm1:lac} and the already well-known consequences in the literature. 

\begin{theorem}\cite[Corollary 2.15]{LMO}
\label{LMO}
	Fix $\vec{r}=(r_{1},r_{2},r_{3})$, with $r_{i}\geq 1$ and $\sum^{3}_{i=1}\frac{1}{r_{i}}>1$, and a sparsity constant $\eta\in (0,1)$. Let $T$ be an operator so that for every $f_{1},f_{2},h\in C^{\infty}_{c}(\mathbb{R}^{n})$
	\begin{equation*}
	\Big|\int_{\mathbb{R}^{n}}T(f_1,f_2)(x) h(x) \,dx\Big|\lesssim \sup_{\mathcal{S}}\Lambda_{\mathcal{S},\vec{r}}(f_{1},f_{2},h),
	\end{equation*}
	where the supremum runs over all sparse families with sparsity constant $\eta$. Then for all exponents $\vec{q}=(q_{1},q_{2})$, with $r_{i}<q_{i}$ for $i=1,2$ and $r'_{3}>q$ and all the weights $\vec{v}=(v_{1},v_{2})\in A_{\vec{q},\vec{r}}$, and for all $f_{1},f_{2},h\in C^{\infty}_{c}(\mathbb{R}^{n})$, we have
$$
	\| T(f_{1},f_{2})\|_{L^{q}(v)}\lesssim \prod^{2}_{i=1}\| f_{i}\|_{L^{q_{i}}(v_{i})},
$$
	where $\frac{1}{q}=\frac{1}{q_{1}}+\frac{1}{q_{2}}$ and $v=v^{\frac{q}{q_{1}}}_{1}v^{\frac{q}{q_{2}}}_{2}$. 
	\end{theorem}
In view of the above theorem, the sparse domination result contained in Theorem~\ref{mainthm1:lac} yields the weighted estimates in Theorem \ref{lacweighted}. 

\subsection{Quantitative bounds in Theorem \ref{lacweighted}}

In \cite{Nie}, an improvement of the quantitative bounds obtained from sparse domination in multilinear forms was achieved. Indeed, the results in \cite{LMO} missed the quantitative weighted bounds for the range $q<1$. This range was accomplished in \cite{Nie}. 

\begin{theorem} \cite[Corollary 4.2]{Nie} Let $T$ be a bilinear or positive valued bi-sublinear operator and assume that for some $p_{1},p_{2}\in (0,\infty)$ and $t\in [1, \infty]$, we have the sparse domination of the bilinear operator for every $f_{1},f_{2},h\in C^{\infty}_{c}(\mathbb{R}^{n})$, i.e., 
\begin{equation*}
	\Big|\int_{\mathbb{R}^{n}}T(f_1,f_2)(x) h(x) \,dx\Big|\lesssim \sup_{\mathcal{S}}\Lambda_{\mathcal{S},(p_1,p_2,t)}(f_{1},f_{2},h),
\end{equation*}
then for all $\vec{q}=(q_{1},q_{2})$ with $q_{1},q_{2}\in (0,\infty]$ such that $(\vec{p},t')< \vec{q}$ and all weights $\vec{w}\in A_{\vec{q},(\vec{p},t')}$, the operator $T$ extends to a bounded operator $L^{q_1}(w^{q_1}_1)\times L^{q_2}(w_2^{q_2}) \rightarrow L^{q}(w^{q})$, where $\frac{1}{q}=\frac{1}{q_{1}}+\frac{1}{q_{2}}$, with the bound
\begin{equation*}
\|T\|_{L^{q_1}(w_1^{q_1})\times L^{q_2}(w_2^{q_2}) \rightarrow L^{q}(w^{q})}\lesssim [\vec{w}]_{A_{\vec{q},(\vec{p},t')}}^{\max\Big(\frac{\frac{1}{p_1}}{\frac{1}{p_1}-\frac{1}{q_1}},\frac{\frac{1}{p_2}}{\frac{1}{p_2}-\frac{1}{q_2}},\frac{1-\frac{1}{t'}}{\frac{1}{q}-\frac{1}{t'}}\Big)}.
\end{equation*}
\end{theorem}

In view of the theorem above, the sparse domination results obtained in Theorem~\ref{mainthm1:lac} yield the following improved weighted estimates for the operators $\mathcal M_{\operatorname{lac}}$ and $\mathcal M_{\operatorname{full}}$.
 
\begin{theorem} \label{lacweightedNie}
Let $n\geq 2$ and $(\frac{1}{r_{i}},\frac{1}{s_{i}})$, $i=1,2,$ be in the interior of $L_n$ (respectively $F_n$). Assume that
$\frac{1}{r_{1}}+\frac{1}{r_{2}}<1$ and $t=\frac{s_{1}s_{2}}{s_{1}+s_{2}-s_{1}s_{2}}>1$. Then for all $\vec{q}>(\vec{r}, t')$ the operator $\mathcal M_{\operatorname{lac}}$ (respectively $\mathcal M_{\operatorname{full}}$) extends to a bounded operator $L^{q_1}(w^{q_1}_1)\times L^{q_2}(w_2^{q_2}) \rightarrow L^{q}(w^{q})$, where $\frac{1}{q}=\frac{1}{q_{1}}+\frac{1}{q_{2}}$,  with the bound
\begin{equation*}
\|\mathcal M\|_{L^{q_1}(w_1^{q_1})\times L^{q_2}(w_2^{q_2}) \rightarrow L^{q}(w^{q})}\lesssim [\vec{w}]_{A_{\vec{q},(\vec{r},t')}}^{\max\Big(\frac{\frac{1}{r_1}}{\frac{1}{r_1}-\frac{1}{q_1}},\frac{\frac{1}{r_2}}{\frac{1}{r_2}-\frac{1}{q_2}},\frac{1-\frac{1}{t'}}{\frac{1}{q}-\frac{1}{t'}}\Big)},
\end{equation*}
where $\mathcal M:=\mathcal M_{\operatorname{lac}}$ (respectively $\mathcal M_{\operatorname{full}}$) and $A_{\vec{q},(\vec{r},t')}$ with $(\vec{r},t')=(r_1,r_2,t')$ is defined as in Definition \ref{wNier}.
\end{theorem}

\begin{remark}
\label{rem:vv}
Note that the end-point extrapolation results in~\cite{LMO1, Nie} allow the index $q_j$ in the theorem above to take value infinity. Moreover, the original Theorem \ref{LMO} contained in \cite{LMO} includes vector-valued results. These apply to our sparse domination in Theorem~\ref{lacweighted}, so that vector-valued inequalities are immediately obtained from \cite[Corollary 2.15]{LMO}, see also \cite[Corollary 4.6]{Nie}.

\end{remark}

\subsection{Weighted boundedness for the triplet $(2,2,1)$}

In this section we present the proof of Theorem~\ref{radial} and, as explained, such a proof will give Proposition \ref{remark} as a by-product. We shall use the ideas from \cite{Calderon, JSK,XQ} in order to prove our theorem.

\begin{proof}[Proof of Theorem \ref{radial}]
We present the proof of the theorem for the operator $\mathcal{M}_{\operatorname{lac}}$. The case of the operator $\mathcal{M}_{\operatorname{full}}$ may be dealt with using the similar ideas with appropriate modifications. The proof of Theorem~\ref{radial} is done in two steps. The first step is to establish a more general result by using analytic interpolation for a family of bilinear operators. Then in the second step we use this general result with a suitable choice of exponents to deduce the theorem. 

{\bf Step I:} Let $1<p_{1},p_{2}<\infty$, $\frac{1}{p}=\frac{1}{p_1}+\frac{1}{p_2},$ and $(\frac{1}{r_{i}},\frac{1}{s_{i}})\in L_{n}$, $i=1,2$ with $\frac{1}{r_{1}}+\frac{1}{r_{2}}<1.$ Write $t=\frac{s_{1}s_{2}}{s_{1}+s_{2}-s_{1}s_{2}}.$ For $\vec{r}=(r_{1},r_{2},t)<\vec{p}:=(p_1,p_2,p)$, let $\vec{w}=(w_{1},w_{2})\in A_{\vec{p},\vec{r}}.$ By Theorem \ref{lacweighted}  we have
 \begin{equation}
 \label{eqs}
 \|\mathcal{M}_{\operatorname{lac}}(f_{1},f_{2})\|_{L^{p}(w)}\leq C_{1}\| f_{1}\|_{L^{p_{1}}(w_{1})}\| f_{2}\|_{L^{p_{2}}(w_{2})}.
 \end{equation}
 Also, note that by Theorem \ref{product} we have the following weighted estimates for the product weights. 
 \begin{equation}\label{eqp}
 \|\mathcal{M}_{\operatorname{lac}}(f_{1},f_{2})\|_{L^{q}(v)}\leq C_{2}\| f_{1}\|_{L^{q_{1}}(v_{1})}\| f_{2}\|_{L^{q_{2}}(v_{2})},
 \end{equation}
 for $1<q_{i}<\infty$, $\frac{1}{q}=\frac{1}{q_1}+\frac{1}{q_2}, v_{i}\in A_{\frac{q_{i}}{t_{i}}}\cap \operatorname{RH}_{\big(\frac{\phi'_{\operatorname{lac}}(\frac{1}{t_{i}})}{q_{i}}\big)'}$, $v=v_{1}^{\frac{q}{q_{1}}}v_{2}^{\frac{q}{q_{2}}}$, $q<1$ and $(\frac{1}{t_{i}},\frac{1}{\eta_{i}})\in L_{n}$ for some $\eta_i\in (1,\infty)$ and $1<t_{i}<q_{i}$, for $i=1,2$. 

We consider the linearised operator $\mathcal{M}_{\operatorname{lac}}$ as follows 
 $$
 \mathcal{M}_{\operatorname{lac}}(f_{1},f_{2})(x)=\mathcal{A}_{\tau(x)}f_{1}(x)\mathcal{A}_{\tau(x)}f_{2}(x),$$
 where $\tau$ is a measurable function from $\mathbb{R}^{n}$ to $[0,\infty)$. For $z\in \mathcal{S}:=\{z\in\mathbb{C}: 0\leq \operatorname{Re}(z)\leq 1\}$, consider the functions 
  \begin{equation*}
  \frac{1}{l(z)}:=\frac{1-z}{p}+\frac{z}{q}, \qquad \frac{1}{l_i(z)}:=\frac{1-z}{p_{i}}+\frac{z}{q_{i}},\qquad i=1,2. 
  \end{equation*}
Choose  $\theta\in (0,1)$ such that  
 \begin{equation*}
\frac{1}{l(\theta)}:=\frac{1-\theta}{p}+\frac{\theta}{q}=1, \qquad \frac{1}{l_i(\theta)}:=\frac{1-\theta}{p_{i}}+\frac{\theta}{q_{i}}=\frac{1}{2}, \qquad i=1,2.
\end{equation*}
Let $k\in(0,1)$ be a number such that $\frac{k}{p}+\frac{k}{q}<1$. 
Note that, for any linear operator $T$ we can write the following. 
 $$
 \| Tf\|^{k}_{L^{1}(\R^n)}=\| |Tf|^{k}\|_{L^{\frac{1}{k}}(\R^n)}=\sup_{\substack{ g\in L^{\frac{1}{1-k}}(\R^n)\\ \| g\|_{L^{\frac{1}{1-k}}(\mathbb{R}^{n})}=1}}\Big|\int |Tf|^{k}g\, dx\Big|.
 $$
 Consider 
 \begin{align*}
 \widetilde{v}_{N}(x)&=v(x),\quad \text{ if } \quad  v(x)\leq N \quad \text{ and } \quad 
 	 \widetilde{v}_{N}(x)=N, \quad \text{ if } \quad  v(x)>N,\\
 \widetilde{w}_{N}(x)&=w(x), \quad \text{ if } \quad w(x)\leq N \quad \text{ and } \quad   \widetilde{w}_{N}(x)=N, \quad \text{ if } \quad w(x)>N. 
 	\end{align*}
Let $f_{1},f_{2}$ be finite simple functions and $g$ be a non-negative finite simple function such that $\| f_{i}\|_{L^{2}(\mathbb{R}^{n})}=1$, for $i=1,2$, and $\| g\|_{L^{\frac{1}{1-k}}(\mathbb{R}^{n})}=1$.\\
With the notations introduced as above, consider the following function.
\begin{equation}\label{psi}
 \psi(z):=\int_{\mathbb{R}^{n}}\big|\mathcal{A}_{\tau(x)}f_{1,z}(x)\mathcal{A}_{\tau(x)}f_{2,z}(x) \widetilde{v}^{\frac{z}{q}}_{N}(x) \widetilde{w}^{\frac{(1-z)}{p}}_{N}(x) g^{\frac{1-\frac{k}{l(z)}}{(1-k)k}}\big|^{k}dx,
 \end{equation}
where 
 \begin{equation*}
 f_{j,z}:=|f_{j}|^{\frac{2}{l_j(z)}}e^{i u_{j}}(v_{j}+\epsilon)^{-\frac{z}{q_{j}}}(w_{j}+\epsilon)^{\frac{z-1}{p_{j}}},\quad j=1,2 	\end{equation*}
 for $z\in \mathcal{S}$, $\epsilon>0$ and $u_{j}\in [0,2\pi]$. Note that we have the following expression for $\psi(\theta)$, $\theta \in (0,1)$, 
 \begin{equation*}
 \psi(\theta)=\int_{\mathbb{R}^{n}}\big| \prod_{i=1}^2  \mathcal{A}_{\tau(x)}(f_{i}(v_{i}+\epsilon)^{-\frac{\theta}{q_{i}}}(w_{i}+\epsilon)^{\frac{\theta-1}{p_{i}}})(x)  \widetilde{v}^{\frac{\theta}{q}}_{N} \widetilde{w}^{\frac{1-\theta}{p}}_{N}\big|^{k}g(x)dx.
 \end{equation*}
For each $x\in\mathbb{R}^{n}$, the functions  $\mathcal{A}_{\tau(x)}f_{i,z}(x)$, $ \widetilde{v}^{\frac{z}{q}}_{N}(x)$, $ \widetilde{w}^{\frac{1-z}{p}}_{N}(x)$ and $g^{\frac{1-\frac{k}{l(z)}}{(1-k)k}}(x)$ are analytic in the domain $\{z\in\mathbb{C}:0<\operatorname{Re}(z)<1\}$. Therefore the integrand in \eqref{psi} is a continuous and subharmonic function in $z\in\mathcal{S}$. Also, using the H\"older's inequality with exponents  $\frac{p}{k}$ and $\frac{p}{p-k}$, it is easy to see that  $\psi$ is a bounded function. 
Moreover, the H\"older's inequality with exponents $\frac{p}{k}$ and $\frac{p}{p-k}$ and the fact that $\| f_{i}\|_{L^{2}(\mathbb{R}^{n})}=1$, $i=1,2$ and $\| g\|_{L^{\frac{1}{1-k}}(\mathbb{R}^{n})}=1$,  yield that 
$$
 |\psi(it)|\leq C^{k}_{1}.
$$
Similarly, using the H\"older's inequality with exponents $\frac{q}{k}$ and $\frac{q}{q-k},$ we get
 $$
 |\psi(1+i t)|\leq C^{k}_{2}.
 $$
 The constants $C_1, C_2$ are independent of $\epsilon, N$ and $\tau$.
 We invoke the maximum modulus principle for subharmonic functions to deduce that 
 \begin{align*}
 |\psi(\theta)| &= \int_{\mathbb{R}^{n}}\big| \prod_{i=1}^2\mathcal{A}_{\tau(x)}(f_{i}v_{i,\epsilon}^{-\frac{\theta}{q_{i}}}w_{i,\epsilon}^{\frac{\theta-1}{p_{i}}})(x)
   \widetilde{v}_{N}^{\frac{\theta}{q}} \widetilde{w}_{N}^{\frac{1-\theta}{p}}\big|^{k}g(x)dx \\
  & \leq  C^{k(1-\theta)}_{1}C^{k\theta}_{2}.
 \end{align*}
Here we have used the notation $v_{i,\epsilon}=v_i+\epsilon$ and  $w_{i,\epsilon}=w_i+\epsilon$ for $i=1,2.$ 
Therefore, using a duality argument we obtain that 
\begin{equation*}
\int_{\mathbb{R}^{n}}\big| \mathcal{A}_{\tau(x)}(f_{1}v_{1,\epsilon}^{-\frac{\theta}{q_{1}}}w_{1,\epsilon}^{\frac{\theta-1}{p_{1}}})(x)
\mathcal{A}_{\tau(x)}(f_{2}v_{2,\epsilon}^{-\frac{\theta}{q_{2}}}w_{2,\epsilon}^{\frac{\theta-1}{p_{2}}})(x)\big|  \widetilde{v}_{N}^{\frac{\theta}{q}} \widetilde{w}_{N}^{\frac{1-\theta}{p}}dx
\leq C \Big(\int_{\mathbb{R}^{n}}|f_{1}|^{2} \Big)^{\frac{1}{2}}\Big(\int_{\mathbb{R}^{n}}|f_{2}|^{2} \Big)^{\frac{1}{2}}.
\end{equation*}
Since the set of finite simple functions is dense in $L^{s}(\R^n),~1\leq s<\infty$, we get the estimate above for all $L^{2}(\R^n)$ functions $f_1$ and $f_2$. Next, recall that the constants $C_{1},C_{2}$ are independent of $\epsilon$, $N$ and $\tau$. Let $\epsilon\rightarrow 0$ and $N\rightarrow\infty$
and replace $f_{i}$ by $f_{i}v^{\frac{\theta}{q_{i}}}_{i}w^{\frac{1-\theta}{p_{i}}}_{i}$, $i=1,2,$ in the above to get that 
  \begin{multline*}
 \int_{\mathbb{R}^{n}}\big|\mathcal{A}_{\tau(x)}f_{1}(x)\mathcal{A}_{\tau(x)}f_{2}(x)\big|v^{\frac{\theta}{q}}(x)w^{\frac{1-\theta}{p}}(x)dx
 \\ \leq C \Big(\int_{\mathbb{R}^{n}}|f_{1}|^{2}v^{\frac{2\theta}{q_{1}}}_{1}w^{\frac{2(1-\theta)}{p_{1}}}_{1} \Big)^{\frac{1}{2}}\Big(\int_{\mathbb{R}^{n}}|f_{2}|^{2}v^{\frac{2\theta}{q_{2}}}_{2}w^{\frac{2(1-\theta)}{p_{2}}}_{2} \Big)^{\frac{1}{2}}.
\end{multline*}
Again, since the above constant $C$ in independent of $\tau$, we get the boundedness of the operator $\mathcal{M}_{\operatorname{lac}}$. 
\begin{multline}\label{interpolation}
\int_{\mathbb{R}^{n}}\big|\mathcal{M}_{\operatorname{lac}}(f_{1},f_{2})(x)\big|v^{\frac{\theta}{q}}(x)w^{\frac{1-\theta}{p}}(x)dx
	\\ \leq C \Big(\int_{\mathbb{R}^{n}}|f_{1}|^{2}v^{\frac{2\theta}{q_{1}}}_{1}w^{\frac{2(1-\theta)}{p_{1}}}_{1} \Big)^{\frac{1}{2}}\Big(\int_{\mathbb{R}^{n}}|f_{2}|^{2}v^{\frac{2\theta}{q_{2}}}_{2}w^{\frac{2(1-\theta)}{p_{2}}}_{2} \Big)^{\frac{1}{2}}.
\end{multline}

{\bf Step II:} We will use the estimate~\eqref{interpolation} above for radial weights with a suitable choice of exponents to conclude the proof of Theorem~\ref{radial} for the case of lacunary operator $\mathcal{M}_{\operatorname{lac}}$. 	

We make the following choice of exponents. For $\epsilon>0,$ let $p_{i}=2+2\epsilon$, $r_{i}=2+\epsilon$ and $(\frac{1}{r_{i}},\frac{1}{s_{i}})\in L_{n}$, $i=1,2$. As earlier we write $t=\frac{s_{1}s_{2}}{s_{1}+s_{2}-s_{1}s_{2}}$  and set $\vec{r}=(r_{1},r_{2},r_{3})$ with $r_{3}=t$. Let $\vec{w}=(|x|^{\alpha'},|x|^{\beta'})\in A_{\vec{p},\vec{r}}$ and note that the estimate~\eqref{eqs} holds for the bilinear weights $A_{\vec{p},\vec{r}}$. Next, for a small positive real number $\delta$, consider $q_i=2-\delta$, $i=1,2,$ and $\vec{v}=(|x|^a, |x|^b)$ with $1-n\leq a,b<(n-1)(1-\delta)$, then we know that the estimate~\eqref{eqp} holds for the operator $\mathcal{M}_{\operatorname{lac}}$. Therefore, by the previous step, the operator $\mathcal{M}_{\operatorname{lac}}$ satisfies the inequality~\eqref{interpolation} for the choice of exponents and weights considered above, i.e., we have 
\begin{multline*}
\int_{\mathbb{R}^{n}}\big|\mathcal{M}_{\operatorname{lac}}(f_{1},f_{2})(x)\big| |x|^{\frac{(a+b)\theta}{2-\delta}+\frac{(\alpha'+\beta')(1-\theta)}{2+2\epsilon}}dx
	\\ \leq C \Big(\int_{\mathbb{R}^{n}}|f_{1}|^{2} |x|^{\frac{2a\theta}{2-\delta}+\frac{2\alpha'(1-\theta)}{2+2\epsilon}}\Big)^{\frac{1}{2}}\Big(\int_{\mathbb{R}^{n}}|f_{2}|^{2}|x|^{\frac{2b\theta}{2-\delta}+\frac{2\beta'(1-\theta)}{2+2\epsilon}} \Big)^{\frac{1}{2}}.
\end{multline*}
with $\theta\in (0,1)$ such that 
$\frac{1-\theta}{2+2\epsilon}+\frac{\theta}{2-\delta}=\frac{1}{2}.$ This implies that $\theta=\frac{\epsilon(2-\delta)}{2\epsilon+\delta}$.
Now, we show that the exponents of weights in the estimate above may be chosen suitably so that they satisfy the hypothesis of Theorem~\ref{radial}. Observe that by Lemma~\ref{weightrelation} we have that $\vec{w}=(|x|^{\alpha'},|x|^{\beta'})\in A_{\vec{p},\vec{r}}$ implies that  
$$|x|^{\frac{\alpha' \theta_{1}}{2+2\epsilon}} \in A_{(\frac{1-r}{r})\theta_{1}}, \quad |x|^{\frac{\beta' \theta_{2}}{2+2\epsilon}} \in A_{(\frac{1-r}{r})\theta_{2}},\quad ~\text{and}~|x|^{\frac{(\alpha'+\beta')\delta_{3}}{2+2\epsilon}}\in A_{\frac{1-r}{r}\delta_{3}},$$
where 

\begin{equation*} 
\frac{1}{\delta_{i}}=\frac{1}{r_{i}}-\frac{1}{p_{i}} \quad 
\text{ and } \quad 
\frac{1}{\theta_{i}}=\frac{1-r}{r}-\frac{1}{\delta_{i}},\quad  i=1,2,3.
\end{equation*}
Substituting values of various parameters, we obtain  that
\begin{equation*}
\frac{1}{r_{3}}=\frac{1}{t}=\frac{2(n-1)+n\epsilon}{n(2+\epsilon)},\qquad 
\frac{1}{r}=\sum^{3}_{i=1}\frac{1}{r_{i}}=\frac{4n+n\epsilon-2}{n(2+\epsilon)}
\end{equation*}
and
\begin{equation*}
\frac{1}{\theta_{i}}=\frac{n(4+3\epsilon)-4(1+\epsilon)}{2n(1+\epsilon)(2+\epsilon)}, \quad i=1,2.
\end{equation*}	
It is easy to verify that $|x|^{\frac{\alpha'\theta_{1}}{2+2\epsilon}}\in A_{(\frac{1-r}{r})\theta_{1}}$ and $|x|^{\frac{\beta' \theta_{2}}{2+2\epsilon}}\in A_{(\frac{1-r}{r})\theta_{2}}$ imply that  
$$
\frac{4(1-n)+4\epsilon-3n\epsilon}{2+\epsilon}<\alpha', \beta'<\frac{n\epsilon}{2+\epsilon}.
$$
Since $\epsilon$ can be chosen arbitrarily small, we get that $\alpha'$ and $\beta'$ satisfy $2(1-n)<\alpha',\beta'<0$. In a similar way,  $|x|^{\frac{(\alpha'+\beta')\delta_{3}}{2+2\epsilon}}\in A_{\frac{1-r}{r}\delta_{3}}$ implies that $\alpha'+\beta'>2(1-n)$. 
Notice that $\theta\rightarrow 0$ as $\epsilon\rightarrow 0$ ($\delta$ is fixed). Since the range of $\alpha'$ and $\beta'$ is an open set, we get that $\mathcal{M}_{\operatorname{lac}}$ is bounded from $L^{2}(|x|^{\alpha})\times L^{2}(|x|^{\beta})$ to $L^{1}(|x|^{\frac{\alpha+\beta}{2}})$ for $\alpha,\beta$ satisfying 
	\begin{equation*}
	2(1-n)<\alpha,\beta<0\quad \text{and} \quad \alpha+\beta>2(1-n).
	\end{equation*}
Further, using the product-type weighted boundedness of $\mathcal{M}_{\operatorname{lac}}$, we get that $\mathcal{M}_{\operatorname{lac}}$ is bounded from $L^{2}(|x|^{a})\times L^{2}(|x|^{b})\rightarrow L^{1}(|x|^{\frac{a+b}{2}})$ for $1-n\leq a,b<n-1$. 
This proves the desired result for the operator $\mathcal{M}_{\operatorname{lac}}$. 
	
The proof for $\mathcal{M}_{\operatorname{full}}$ may be completed in a similar fashion with the extra conditions that $(\frac{1}{r_{i}},\frac{1}{s_{i}})\in F_{n}$, $r_{1},r_{2}>\frac{n}{n-1}$ and $n \geq 3$. Observe that here the restriction on the dimension arises for the case of the full spherical maximal operator due to the estimate~\eqref{eqp}, where the $L^2$ boundedness of $M_{\operatorname{full}}$ is required (following from Theorem \ref{product}). Indeed, $M_{\operatorname{full}}$ is not bounded for $p\le 2$ in dimension $n=2$.
\end{proof} 
%
\section{Comparing Theorem \ref{lacweighted} with H\"{o}lder type results} 
\label{compar}
For $1<p<\infty$, define the sets 
$ \mathcal L_p:=\{ w: M_{\operatorname{lac}}\,\,\text{maps} \,\, L^p(w)~\text{to}~L^p(w)\}$ and 
$\mathcal F_p:=\{ w: M_{\operatorname{full}}\,\, \text{maps} \,\, L^p(w)~\text{to}~L^p(w)\}.$ 
Recall also the definitions of $\mathcal{R}_{p}$ and $\mathcal{\widetilde{R}}_{p}$ in \eqref{Rp} and \eqref{Rpt}, respectively. In~\cite{DV}, Duoandikoetxea and Vega proved the following weighted estimates for spherical maximal functions with respect to radial weights. 
\begin{theorem}\label{Duo-Vega}\cite{DV} 
 $\mathcal{R}_{p} \subseteq \mathcal L_p$,  $1<p<\infty$ and 
 $\mathcal{\widetilde{R}}_{p} \subseteq \mathcal F_p$, $\frac{n}{n-1}<p<\infty.$
\end{theorem}

Recently, in~\cite{Lacey} Lacey proved the weighted estimates for the operators with respect to general weights using sparse domination principle. 
\begin{theorem}
\label{lacey}\cite{Lacey} The following estimates hold. 
\begin{itemize}
\item Let $1<r<p<\phi_{\operatorname{lac}}'(\frac{1}{r})$, then $A_{\frac{p}{r}} \cap \operatorname{RH}_{\big(\frac{\phi_{\operatorname{lac}}'(\frac{1}{r})}{p}\big)'}\subseteq \mathcal L_p.$
\item Let $\frac{n}{n-1}<r<p<\phi_{\operatorname{full}}'(\frac{1}{r})$, then $A_{\frac{p}{r}}\cap \operatorname{RH}_{\big(\frac{\phi'_{\operatorname{full}}(\frac{1}{r})}{p}\big)'}\subseteq \mathcal F_p.$
\end{itemize}
\end{theorem}
For $\vec p=(p_1, p_2,p)$ with $\frac{1}{p}=\frac{1}{p_1}+\frac{1}{p_2}$, define 
$$
\mathcal{L}_{\vec{p}}:=\big\{ \vec{w}=({w}_1, {w}_2): \mathcal M_{\operatorname{lac}}~\text{ maps }~L^{p_1}(w_1)\times L^{p_2}(w_2)~\text{to}~ L^{p}(w)\big\} 
$$ 
and
$$
\mathcal{F}_{\vec{p}}:=\big\{ \vec{w}=({w}_1, {w}_2): \mathcal M_{\operatorname{full}}~\text{ maps }~L^{p_1}(w_1)\times L^{p_2}(w_2)~\text{to}~ L^{p}(w)\big\}.
$$ 
In view of Theorems~\ref{Duo-Vega} and \ref{lacey}, H\"{o}lder's inequality yields the following weighted estimates for bilinear spherical maximal functions with respect to product type bilinear weights. 
\begin{theorem} \label{product}
The following holds:
\begin{itemize}
\item $\prod_{i=1}^2 A_{\frac{p_i}{r_i}}\cap \operatorname{RH}_{\big(\frac{\phi_{\operatorname{lac}}'(\frac{1}{r_i})}{p_i}\big)'}\subseteq \mathcal{L}_{\vec{p}}$ for all $1<r_i<p_i<\phi_{\operatorname{lac}}'(\frac{1}{r_i})$ and $\prod_{i=1}^2 \mathcal{R}_{p_i} \subseteq \mathcal{L}_{\vec{p}}, $ where $p_i>1$, $i=1,2.$
\item $\prod_{i=1}^2 A_{\frac{p_i}{r_i}}\cap \operatorname{RH}_{\big(\frac{\phi_{\operatorname{full}}'(\frac{1}{r_i})}{p_i}\big)'}\subseteq \mathcal{F}_{\vec{p}}$ for all $\frac{n}{n-1}<r_i<p_i<\phi_{\operatorname{full}}'(\frac{1}{r_i})$ and $\prod_{i=1}^2 \mathcal{\widetilde{R}}_{p_i} \subseteq \mathcal{F}_{\vec{p}}, $ where $p_i>\frac{n}{n-1}$, $i=1,2.$
\end{itemize} 
\end{theorem}
In this section we show that Theorem~\ref{lacweighted} addresses the weighted boundedness of bilinear operators $\mathcal M_{\operatorname{lac}}$ and $\mathcal M_{\operatorname{full}}$ with respect to bilinear weights that are not of product type as covered by Theorem~\ref{product} above. 

\subsection{The case of lacunary spherical maximal operator $\mathcal{M}
_{\operatorname{lac}}$.} 

Let $n\geq 2$.  Consider $\vec{p}=(p_{1},p_{2},p)$ and $\vec{r}=(r_{1},r_{2},t)$, where $p_1=p_2=n+\delta$, $\delta>1$  and $r_{1}=r_{2}=2+\epsilon$, $\epsilon>0$, $t=\frac{s_{1}s_{2}}{s_{1}+s_{2}-s_{1}s_{2}}$ for $(\frac{1}{r_{i}},\frac{1}{s_{i}})\in L_{n}$, $i=1,2$. Note that $p=\frac{n+\delta}{2}$ and $p_{3}=p'=\frac{n+\delta}{n+\delta-2}$. 
With this choice of exponents,  let $\vec{w}=(w_{1},w_{2})=(|x|^{a},|x|^{b}) \in A_{\vec{p},\vec{r}}.$ Note that  Theorem~\ref{lacweighted} is applicable for the bilinear weight $\vec{w}$. Moreover, the condition $\vec{r}\leq \vec{p}$ implies that $\epsilon<n-2+\delta$ and $\frac{\delta-n}{n}<\epsilon$. Therefore,  for the validity of Theorem~\ref{lacweighted}, we see that $\epsilon$ can vary between $\max\{0,\frac{\delta-n}{n}\}$ and $n-2+\delta.$ 
Next, invoking Lemma~\ref{weightrelation} we know that $\vec{w} \in A_{\vec{p},\vec{r}}$ if, and only if $w^{\frac{\theta_{i}}{p_{i}}}_{i}\in A_{\frac{1-r}{r}\theta_{i}}$, $i=1,2$ and $w^{\frac{\delta_{3}}{p}}\in A_{\frac{1-r}{r}\delta_{3}}$, where 
\begin{equation*}
\frac{1}{r}=\sum_{i=1}^3\frac{1}{r_{i}}, \quad 
\frac{1}{\theta_{i}}=\frac{1-r}{r}-\frac{1}{\delta_{i}},\quad 
\frac{1}{\delta_{i}}=\frac{1}{r_{i}}-\frac{1}{p_{i}},\quad ~\text{for }i=1,2,3~ \text{with}~r_3=t.
\end{equation*}
Substituting the values of various parameters we get $\theta_{1}=\theta_{2}=\frac{n(n+\delta)(2+\epsilon)}{(n-2)(n+\delta)+n(2+\epsilon)}$. Moreover, $w^{\frac{\theta_{1}}{p_{1}}}_{1}\in A_{\frac{1-r}{r}\theta_{1}}$ gives us the possible range of exponent $a$, which is 
\begin{equation}\label{eq50}
a\in\Big(-n-\frac{(n-2)(n+\delta)}{2+\epsilon},\frac{n(n+\delta-2-\epsilon)}{2+\epsilon}\Big).
\end{equation}
Next, we shall compute the possible range of exponents for product type bilinear weights for the triplet $\vec{p}$ and compare it with the range of $a$ given above. 
	
For, let $\vec{p}$ be as above and take $t_{1}=t_{2}=n+\delta-\alpha$ for some $\alpha>0$.
Let $(|x|^{a},|x|^{b})\in \prod_{i=1}^2 A_{\frac{p_{i}}{t_{i}}}\cap \operatorname{RH}_{\big(\frac{\phi'_{\operatorname{lac}}(\frac{1}{t_{i}})}{p_{i}}\big)'}$ be a product type bilinear weight. We discuss the two cases depending on the definition of the function $\phi_{\operatorname{lac}}$ in \eqref{philac} separately. 

\textbf{Case 1}: If $\frac{1}{\phi_{\operatorname{lac}}(\frac{1}{t_{i}})}=1-\frac{1}{nt_{i}}$, then $\phi_{\operatorname{lac}}'(\frac{1}{t_{i}})=nt_{i}=n(n+\delta-\alpha)$. Note that $\phi_{\operatorname{lac}}'(\frac{1}{t_{i}})>n+\delta$ which in turn implies that  $\alpha<\frac{(n-1)(n+\delta)}{n}.$
We know that 	
\begin{align}
 |x|^{a}\in A_{\frac{n+\delta}{n+\delta-\alpha}}\cap \operatorname{RH}_{\big(\frac{\phi'_{\operatorname{lac}}(\frac{1}{n+\delta-\alpha})}{n+\delta}\big)'}
\nonumber &\Leftrightarrow  |x|^{a}\in A_{\frac{n+\delta}{n+\delta-\alpha}}\cap \operatorname{RH}_{\frac{n(n+\delta-\alpha)}{n(n+\delta-\alpha)-(n+\delta)}}\\
\nonumber &\Leftrightarrow |x|^{\frac{an(n+\delta-\alpha)}{n(n+\delta-\alpha)-(n+\delta)}}\in A_{\frac{n(n+\delta-\alpha)}{n(n+\delta-\alpha)-(n+\delta)}(\frac{n+\delta}{n+\delta-\alpha}-1)+1}\\
\label{eq51}&\Leftrightarrow  a\in\Big(-n+\frac{n+\delta}{n+\delta-\alpha},\frac{n\alpha}{n+\delta-\alpha}\Big).
\end{align}
Comparing the estimate (\ref{eq51}) with the range of exponent $a$ given by (\ref{eq50}), we see that (\ref{eq50}) allows values of $a$ which are not possible in the product type weights. 

\textbf{
Case 2}: Since $p_1=p_2>n+1,$ we claim that the value $\frac{1}{\phi_{\operatorname{lac}}(\frac{1}{t_{i}})}=n(1-\frac{1}{t_{i}})$ is not possible. For,  take $\frac{1}{\phi_{\operatorname{lac}}(\frac{1}{t_{i}})}=n(1-\frac{1}{t_{i}})$, then it would imply that $\frac{n}{n+1}<\frac{1}{t_{i}}=\frac{1}{n+\delta-\alpha}$. From here we get that $\alpha>\frac{n^{2}+n\delta-n-1}{n}.$ 
On the other hand the condition $\phi'_{\operatorname{lac}}(\frac{1}{n+\delta-\alpha})>n+\delta$ implies that 
\begin{equation*}
\frac{n+\delta-\alpha}{n-(n+\delta-\alpha)(n-1)}>n+\delta.
\end{equation*}
Hence
\begin{equation*}
 \frac{(n-1)(n+\delta)(n+\delta-1)}{(n-1)(n+\delta)+1} >\alpha.
\end{equation*}
It is easy to verify that for $\delta>1$ the estimate above contradicts the earlier estimate $\alpha>\frac{n^{2}+n\delta-n-1}{n}.$ 
This establishes the claim that Theorem~\ref{lacweighted} provides weighted boundedness of the bilinear lacunary spherical maximal operator $\mathcal{M}_{\operatorname{lac}}$ for bilinear weights which are not product of weights arising from Theorem~\ref{lacey}. 

We also need to verify the claim against the product type weights arising from Theorem~\ref{Duo-Vega}. However, it is easier to verify this claim as we know that $(|x|^{a},|x|^{b})\in\mathcal{R}_{n+\delta}\times\mathcal{R}_{n+\delta}$ would imply that $a,b\in [1-n,(n-1)(n+\delta-1)\big)$. Comparing this with the range of exponent $a$ given by (\ref{eq50}) proves the claim. 
\qed

Next, we discuss the case of the bilinear full spherical maximal operator $\mathcal{M}_{\operatorname{full}}$. It is similar to the previous case and hence we skip details. 

\subsection{The case of full spherical maximal operator $\mathcal{M}_{\operatorname{full}}$.} 

We consider the same setting as in the previous section until the estimate~(\ref{eq50}) with $n\geq 3$. Notice that here we would require that $(\frac{1}{r_{i}},\frac{1}{s_{i}})\in F_{n}$, $i=1,2$. 

Now we compute the range of exponents when $(|x|^{a},|x|^{b})\in \prod_{i=1}^2 A_{\frac{p_{i}}{t_{i}}}\cap \operatorname{RH}_{\big(\frac{\phi'_{\operatorname{full}}(\frac{1}{t_{i}})}{p_{i}}\big)'}$. As earlier we consider two cases separately depending on the function $\phi_{\operatorname{full}}$. 

\textbf{
Case 1}: If  $\frac{1}{\phi_{\operatorname{full}}(\frac{1}{t_{i}})}=1-\frac{1}{nt_{i}}$, then $\phi_{\operatorname{full}}'(\frac{1}{t_{i}})=nt_{i}=n(n+\delta-\alpha)$. A similar computation as in the previous section gives us that $\alpha<\frac{(n-1)(n+\delta)}{n}$
and that $a$ varies over the range $ a\in\big(-n+\frac{n+\delta}{n+\delta-\alpha},\frac{n\alpha}{n+\delta-\alpha}\big).$	Comparing this with (\ref{eq50}) we see that the possible range of exponent $a$ in product type weights does not exhaust all the values of $a$ given by estimate~\eqref{eq50}. 

\textbf{
Case 2}:  Let $\frac{1}{\phi_{\operatorname{full}}(\frac{1}{n+\delta-\alpha})}=2-\frac{n+1}{(n-1)(n+\delta-\alpha)}.$ We show that for $\delta>\frac{n+1}{n-1}$ this choice of $\phi_{\operatorname{full}}$ is not possible. 
Observe that we have $\frac{n^{2}-n}{n^{2}+1}<\frac{1}{n+\delta-\alpha}$ which gives us $ \alpha>\frac{(n^{2}-n)(n+\delta)-(n^{2}+1)}{n^{2}-n}.$ Also,   
\begin{equation*}
\phi'_{\operatorname{full}}\Big(\frac{1}{n+\delta-\alpha}\Big)=\frac{(n-1)(n+\delta-\alpha)}{(n+1)-(n-1)(n+\delta-\alpha)}>n+\delta.
\end{equation*}
This yields another estimate on $\alpha$, i.e., $\alpha<\frac{(n+\delta)[(n+\delta)(n-1)-2]}{(n+\delta+1)(n-1)}.$ Since $\delta>\frac{n+1}{n-1}$ this contradicts the earlier estimate on $\alpha.$ This proves the claim.

Next, if $0<\delta\leq \frac{n+1}{n-1}$, then computing the range of $a$ keeping in mind the estimate on $\phi'_{\operatorname{full}}$, we get that 
\begin{align*}
|x|^{a} \in A_{\frac{n+\delta}{n+\delta-\alpha}}\cap \operatorname{RH}_{\big(\frac{\phi'_{\operatorname{full}}(\frac{1}{n+\delta-\alpha})}{n+\delta}\big)'}
&\Leftrightarrow  -n<a\gamma<\frac{n\gamma\alpha}{n+\delta-\alpha}\\
&\Leftrightarrow  a\in\Big(\frac{-n}{\gamma},\frac{n\alpha}{n+\delta-\alpha}\Big),
\end{align*}
where $\gamma=\frac{(n-1)(n+\delta-\alpha)}{(n-1)(n+\delta-\alpha)-(n+1)(n+\delta)+(n-1)(n+\delta)(n+\delta-\alpha)}.$

Since $\alpha>\frac{(n^{2}-n)(n+\delta)-(n^{2}+1)}{n^{2}-n}$, it is easy to verify that $\gamma>1$. This further implies that the possible range of exponent $a$ in product type weights does not exhaust all the values of $a$ given by the range (\ref{eq50}). 

Finally, the case of $(|x|^{a},|x|^{b})\in \widetilde{\mathcal{R}}_{n+\delta}\times \widetilde{\mathcal{R}}_{n+\delta}$ can be deduced in a similar fashion as in the previous section. This implies that $a,b\in(1-n,(n-1)(n+\delta-1)-1\big)$. Therefore, by comparing it with (\ref{eq50}) we get the assertion.  

\section{Sparse domination: proof of Theorem \ref{mainthm1:lac}}
\label{sec:proof}

In this section we prove Theorem \ref{mainthm1:lac}. In order to prove our results we have exploited the corresponding ideas for the linear case from~\cite{Lacey}. As announced, we will proceed in two steps, proving first a stronger version for characteristic function and later the result for general functions. We follow a unified approach, stating as simultaneously as possible the results for both $\mathcal{M}_{\operatorname{lac}}$ and $\mathcal{M}_{\operatorname{full}}$.

\begin{theorem}\label{mainthm2:lac}
Let $n\geq 2$. For  $i=1,2$, let $(\frac{1}{r_{i}},\frac{1}{s_{i}})$ be in the interior of the triangle $L_n$ (respectively the trapezium $F_n$). Then for characteristic functions $f_{1}=\chi_{F_{1}}$, $f_{2}=\chi_{F_{2}}$ and compactly supported bounded function $h$, where $F_{1},F_{2}$ are bounded measurable subsets of $\R^n$, there exists a sparse collection $\mathcal{S}=\mathcal{S}_{r_{1},r_{2},t}$ such that 
\begin{equation*}\langle \mathcal{M}(f_{1},f_{2}),h\rangle\leq C \Lambda_{\mathcal{S}_{r_{1},r_{2},t}}(f_{1},f_{2},h),
\end{equation*}
where $t=\frac{s_{1}s_{2}}{s_{1}+s_{2}-s_{1}s_{2}}>1$ and $\mathcal{M}:=\mathcal{M}_{\operatorname{lac}}$ (respectively  $\mathcal{M}_{\operatorname{full}}$).
\end{theorem}

For a cube $Q\subset\mathbb{R}^{n}$ with side-length $l_{Q}=2^{q}$, define $\mathcal{A}_{Q}f_{i}(x)=\mathcal{A}_{2^{q-2}}(f_{i}\chi_{\frac{1}{3}Q})(x)$, $i=1,2$. Note that $\mathcal{A}_{Q}f_{i}$ is supported in the cube $Q$. The following lemma involving stopping time arguments is the key result in the proof of Theorem \ref{mainthm2:lac}. 

\begin{lemma}\label{keylac}
Let $n\geq 2$. For  $i=1,2$, let $(\frac{1}{r_{i}},\frac{1}{s_{i}})$ be in the interior of the triangle $L_n$ (respectively the trapezium $F_n$), with the additional condition $\frac{1}{s_{1}}+\frac{1}{s_{2}}>1$. 
 \label{l1} Let $f_{1}=\chi_{F_{1}}$, $f_{2}=\chi_{F_{2}}$, where $F_{1},F_{2}$ are measurable subsets of $Q_{0}$ and $h$ be a bounded function supported in $Q_{0}$. Let $C_{0}>1$ be a constant and let $\mathcal{D}_{0}$ be a collection of dyadic subcubes of $Q_0$ such that
\begin{equation*}
\sup_{Q'\in\mathcal{D}_{0}}\sup_{Q:Q'\subset Q\subset Q_{0}}
\Big(\frac{\langle f_{i}\rangle_{Q,r_{i}}}{\langle f_{i}\rangle_{Q_{0},r_{i}}}+\frac{\langle h\rangle_{Q,t}}{\langle h\rangle_{Q_{0},t}}\Big)\leq C_{0},\ \ \text{for}\ \  i=1,2, 
\end{equation*}
	where $t=\frac{s_{1}s_{2}}{s_{1}+s_{2}-s_{1}s_{2}}>1$.
	Then, 
	\begin{itemize} 
	\item[(i)] If $(\frac{1}{r_{i}},\frac{1}{s_{i}})$ are in the interior of $L_n$, with the additional condition $\frac{1}{s_{1}}+\frac{1}{s_{2}}>1$,
\begin{equation*}
|\langle \sup_{Q\in\mathcal{D}_{0}}\mathcal{A}_{Q}f_{1}\mathcal{A}_{Q}f_{2}, h\rangle|\lesssim  |Q_{0}|\langle f_{1}\rangle_{Q_{0},r_{1}}\langle f_{2}\rangle_{Q_{0},r_{2}}\langle h\rangle_{Q_{0},t}, 
\end{equation*}

\item[(ii)] If $(\frac{1}{r_{i}},\frac{1}{s_{i}})$ are in the interior of $F_n$, with the additional condition $\frac{1}{s_{1}}+\frac{1}{s_{2}}>1$,
\begin{equation*}
|\langle \sup_{Q\in\mathcal{D}_{0}}\sup_{2^{q-3}\leq t\leq 2^{q-2}}\mathcal{A}_{t}(f_{1}\chi_{\frac{1}{3}Q})(x)\mathcal{A}_{t}(f_{2}\chi_{\frac{1}{3}Q}),h\rangle|\lesssim  |Q_{0}|\langle f_{1}\rangle_{Q_{0},r_{1}}\langle f_{2}\rangle_{Q_{0},r_{2}}\langle h\rangle_{Q_{0},t}.
\end{equation*}
\end{itemize}
\end{lemma}

We assume Lemma~\ref{keylac} for a moment and complete the proof of Theorems~\ref{mainthm1:lac} and  \ref{mainthm2:lac}. 

\subsection{Proof of Theorem~\ref{mainthm2:lac}}
We will present the proof for $\mathcal M_{\operatorname{lac}}$, and after that we will point out the main differences in the proof for $\mathcal M_{\operatorname{full}}$.

First note that using standard arguments we can reduce our work to proving analogous results for the dyadic version of the maximal functions under consideration. Indeed, let $f_{1}$ and $f_{2}$ be positive functions with their support contained inside a cube $Q_{0}$. Fix a dyadic lattice $\mathcal{D}$ and consider the maximal function 
$$\mathcal{M}_{\mathcal{D}}(f_{1},f_{2})(x):=\sup_{Q\in\mathcal{D}}|\mathcal{A}_{Q}f_{1}(x)\mathcal{A}_{Q}f_{2}(x)|.$$
Since $\operatorname{supp}(f_{i})\subset Q_{0}$, we get that $\mathcal{A}_{Q}f_{i}=0$ if $Q\cap Q_{0}=\emptyset$ and also $\mathcal{A}_{Q}f_{i}=0$ for large enough cubes. In view of this, it is enough to prove corresponding sparse domination for the bilinear maximal operator $$\mathcal{M}_{\mathcal{D}\cap Q_{0}}(f_{1},f_{2})(x)=\sup_{Q\in\mathcal{D}\cap Q_{0}}|\mathcal{A}_{Q}f_{1}(x)\mathcal{A}_{Q}f_{2}(x)|.$$
Then, $\langle\mathcal{M}_{\operatorname{lac}}(f_{1},f_{2}),h\rangle $ can be dominated by the sum of finitely many sparse forms. Finally, one can find a universal sparse form (see \cite[Proposition 2.1]{JM}) in the sparse domination.

We proceed to prove the sparse domination result for the operator $\mathcal{M}_{\mathcal{D}\cap Q_{0}}$. 
Let $C_0$ be a constant and $\mathcal{E}_{Q_{0}}$ denote the collection of maximal dyadic subcubes of $Q_{0}$ satisfying 
		\begin{equation}
		\label{inequ}
			\langle f_{1}\rangle_{Q,r_{1}}>C_{0}\langle f_{1}\rangle_{Q_{0},r_{1}}\quad \text{or}\quad \langle f_{2}\rangle_{Q,r_{2}}>C_{0}\langle f_{2}\rangle_{Q_{0},r_{2}}\quad \text{or}\quad  \langle h\rangle_{Q,t}>C_{0}\langle h\rangle_{Q_{0},t}.
		\end{equation}
Let $E_{Q_{0}}=\cup_{P\in\mathcal{E}_{Q_{0}}}P$. Note that we can choose $C_{0}>1$ so that $|E_{Q_{0}}|<\frac{1}{2}|Q_{0}|$. Writing $F_{Q_{0}}=Q_{0}\setminus E_{Q_{0}}$, we have that $|F_{Q_{0}}|\geq \frac{1}{2}|Q_{0}|$.

Next, denote $\mathcal{D}_{0}:=\{Q\in\mathcal{D}\cap Q_{0}:Q\cap E_{Q_{0}}=\emptyset\}$ and observe that for  $Q\in\mathcal{D}_{0}$ we get that  
		\begin{equation}
		\label{inequm}
			\langle f_{1}\rangle_{Q,r_{1}}\leq C_{0}\langle f_{1}\rangle_{Q_{0},r_{1}} \quad \text{and}\quad \langle f_{2}\rangle_{Q,r_{2}}\leq C_{0}\langle f_{2}\rangle_{Q_{0},r_{2}}\quad \text{and} \quad \langle h\rangle_{Q,t}\leq C_{0}\langle h\rangle_{Q_{0},t}.
		\end{equation}
For, if \eqref{inequ} holds,
		then there exists $P\in\mathcal{E}_{Q_{0}}$ such that $P\supset Q$. This will contradict the definition of $\mathcal{D}_{0}$. In a similar way, note that if $Q'\in\mathcal{D}_{0}$ and $Q'\subset Q\subset Q_{0}$, then we also have \eqref{inequm}. These two observations together give us that, for  $i=1,2$,
		\begin{equation}
		\label{stph}
\sup_{Q'\in\mathcal{D}_{0}}\sup_{Q:Q'\subset Q\subset Q_{0}}\langle f_{i}\rangle_{Q,r_{i}}\leq C_{0}\langle f_{i}\rangle_{Q_{0},r_{i}} \quad \text{and}~ \quad \sup_{Q'\in\mathcal{D}_{0}}\sup_{Q:Q'\subset Q\subset Q_{0}}\langle h\rangle_{Q,t}\leq C_{0}\langle h\rangle_{Q_{0},t}.
		\end{equation}
		
Now we claim, using a standard linearisation argument, that it is enough to prove sparse domination for a suitable linearised form. For, let $\mathcal{Q}$ be the collection of all dyadic subcubes of $Q_{0}$. Given $Q\in\mathcal{Q}$, consider the set 
$$ H_{Q}:=\Big\{x\in Q: \mathcal{A}_{Q}f_{1}(x)\mathcal{A}_{Q}f_{2}(x)\geq\frac{1}{2}\sup_{P\in\mathcal{Q}}\mathcal{A}_{P}f_{1}(x)\mathcal{A}_{P}f_{2}(x)\Big\}.
	 $$
Note that for any $x\in Q_{0},$ there exists a cube $Q\in\mathcal{Q}$ such that $x\in H_{Q}$. Set $B_{Q}=H_{Q}\setminus\bigcup_{Q'\supseteq Q}H_{Q'}$. Observe that $\{B_{Q}\}_{Q\in\mathcal{Q}}$ are pairwise disjoint and $\bigcup_{Q\in\mathcal{Q}}B_{Q}=\bigcup_{Q\in\mathcal{Q}}H_{Q}$. Then
\begin{align*}
\langle\sup_{P\in\mathcal{Q}} \mathcal{A}_{P}f_{1}\mathcal{A}_{P}f_{2},h\rangle&= \sum_{Q\in\mathcal{Q}}\int_{B_{Q}}\sup_{P\in\mathcal{Q}}\left(\mathcal{A}_{P}f_{1}(x)\mathcal{A}_{P}f_{2}(x)\right)h(x)dx\\
 &\leq  2\sum_{Q\in\mathcal{Q}}\int_{B_{Q}}\mathcal{A}_{Q}f_{1}(x)\mathcal{A}_{Q}f_{2}(x)h(x)dx\\
 &\leq  2\sum_{Q\in\mathcal{Q}}\int_{\mathbb{R}^{n}}\mathcal{A}_{Q}f_{1}(x)\mathcal{A}_{Q}f_{2}(x)h(x)\chi_{B_{Q}}(x)dx\\ 
 &= 2\sum_{Q\in\mathcal{Q}}\langle\mathcal{A}_{Q}f_{1}\mathcal{A}_{Q}f_{2},h_{Q}\rangle, 
\end{align*}
where $h_{Q}=h\chi_{B_{Q}}$.
The estimate above allows us to work with a linearised form instead of the supremum. Notice that this argument uses the full collection of dyadic subcubes of the given cube $Q_0$. Indeed, the  linearisation may be used for the collection of cubes under consideration in the following manner. Note that if $Q\in \mathcal{D}_{0}$ then $Q\subset F_{Q_{0}}$ and we have that 
\begin{align*}
	\big|\langle\sup_{Q\in\mathcal{D}_{0}}\mathcal{A}_{Q}f_{1}\mathcal{A}_{Q}f_{2},h\rangle\big|=\big|\langle\sup_{Q\in\mathcal{Q}}\mathcal{A}_{Q}f_{1}\mathcal{A}_{Q}f_{2},h\chi_{F_{Q_{0}}}\rangle\big|&\leq 2\big|\sum_{Q\in\mathcal{Q}}\langle\mathcal{A}_{Q}f_{1}\mathcal{A}_{Q}f_{2},h\chi_{B_{Q}}\chi_{F_{Q_{0}}}\rangle\big|\\
	&=2\big|\sum_{Q\in\mathcal{D}_{0}}\langle\mathcal{A}_{Q}f_{1}\mathcal{A}_{Q}f_{2},h\chi_{B_{Q}}\rangle\big|.
	\end{align*}
Therefore, it suffices to prove the sparse domination for 
\begin{equation*}
\sum_{Q\in \mathcal{D}\cap Q_{0}}\langle \mathcal{A}_{Q}f_{1}\mathcal{A}_{Q}f_{2},h\chi_{B_{Q}}\rangle.
\end{equation*}	
	Next, observe that for any cube $Q\in\mathcal{D}\cap Q_{0}$ we either have $Q\in\mathcal{D}_{0}$ or $Q\subset P$ for some $P\in\mathcal{E}_{Q_{0}}$. Therefore,  
		\begin{equation}\label{factorization}
	 \sum_{Q\in\mathcal{D}\cap Q_{0}}\langle \mathcal{A}_{Q}f_{1}\mathcal{A}_{Q}f_{2},h\chi_{B_{Q}}\rangle =\sum_{Q\in\mathcal{D}_{0}}\langle \mathcal{A}_{Q}f_{1}\mathcal{A}_{Q}f_{2},h\chi_{B_{Q}}\rangle +\sum_{P\in\mathcal{E}_{Q_{0}}}\sum_{Q\subset P}\langle \mathcal{A}_{Q}f_{1}\mathcal{A}_{Q}f_{2},h\chi_{B_{Q}}\rangle.
		\end{equation}
We would like to remark here that so far we have not required that $f_1$ and $f_2$ are characteristic functions. Now we invoke Lemma~\ref{keylac} to get that  
		\begin{equation}\label{itera1}
		\sum_{Q\in\mathcal{D}_{0}}\langle \mathcal{A}_{Q}f_{1}\mathcal{A}_{Q}f_{2},h\chi_{B_{Q}}\rangle\lesssim |Q_{0}|\langle f_{1}\rangle_{Q_{0},r_{1}}\langle f_{2}\rangle_{Q_{0},r_{2}}\langle h\rangle_{Q_{0},t}.
		\end{equation}
Let $\{P_{j}\}$ be an enumeration of cubes in $\mathcal{E}_{Q_{0}}$. Then we can rewrite the remaining term as 
		\begin{equation*}
	\sum_{P\in\mathcal{E}_{Q_{0}}}\sum_{Q\subset P}\langle \mathcal{A}_{Q}f_{1}\mathcal{A}_{Q}f_{2},h\chi_{B_{Q}}\rangle=		\sum^{\infty}_{j=1}\sum_{Q\in P_{j}\cap\mathcal{D}}\langle \mathcal{A}_{Q}f_{1}\mathcal{A}_{Q}f_{2},h\chi_{B_{Q}}\rangle.
		\end{equation*}
We repeatedly use the estimate above for each $j$ and put all the terms together to get a sparse collection $\mathcal{S}$ so that the following holds
		\begin{equation*}
		\sum_{Q\in\mathcal{D}\cap Q_{0}}\langle \mathcal{A}_{Q}f_{1}\mathcal{A}_{Q}f_{2},h\chi_{B_{Q}}\rangle\lesssim\sum_{S\in\mathcal{S}}|S|\langle f_{1}\rangle_{S,r_{1}}\langle f_{2}\rangle_{S,r_{2}}\langle h\rangle_{S,t}.
\end{equation*}
This completes the proof of Theorem~\ref{mainthm2:lac} for $\mathcal M_{\operatorname{lac}}$. 

In order to prove the corresponding results for the operator $\mathcal M_{\operatorname{full}}$, we require a bilinear analogue of local spherical maximal functions. It is defined as follows
\begin{equation}
\label{supfu}
\widetilde{\mathcal{M}}(f_{1},f_{2})(x):=\sup_{t\in[1,2]}\mathcal{A}_{t}f_{1}(x)\mathcal{A}_{t}f_{2}(x).
\end{equation}
Again standard arguments reduce the task to consider a dyadic version with the maximal function
$$
\sup_{Q\in\mathcal{D}\cap Q_0}|\widetilde{\mathcal{M}}_{Q}(f_{1},f_{2})(x)|,
$$
where
\begin{equation*}
\widetilde{\mathcal{M}}_{Q}(f_{1},f_{2})(x):=\sup_{2^{q-3}\leq t\leq 2^{q-2}}\mathcal{A}_{t}(f_{1}\chi_{\frac{1}{3}Q})(x)\mathcal{A}_{t}(f_{2}\chi_{\frac{1}{3}Q})(x).
\end{equation*}
Note that a linearisation trick as earlier tells us that it suffices to replace the supremum~\eqref{supfu} with the form 
\begin{equation*}
|\sum_{Q\in \mathcal{D}_0}\langle \widetilde{\mathcal{M}}_{Q}(f_{1},f_{2}), h_{Q}\rangle|,
\end{equation*} 
with $h_{Q}=h\chi_{B_{Q}}$ and $B_{Q}=E_{Q}\setminus \bigcup_{Q'\supset Q}E_{Q'} $, where
	 $$E_{Q}=\Big\{x\in Q\in\mathcal{Q}: \widetilde{\mathcal{M}}_{Q}(f_{1},f_{2})(x)\geq \frac{1}{2}\sup_{P\in\mathcal{Q}}\widetilde{\mathcal{M}}_{P}(f_{1},f_{2})(x)\Big\}.$$
The remaining part of the proof can be completed following the lacunary case. 

\subsection{Proof of Theorem \ref{mainthm1:lac}}

We will make use of Theorem~\ref{mainthm2:lac} in proving Theorem~\ref{mainthm1:lac}. The proof is unified in both lacunary and full cases. 

Let $f_{1}, f_{2},$ and $h$ be non-negative compactly supported bounded functions with support in the cube $Q_{0}$. We use the same argument as in the proof of Theorem~\ref{mainthm2:lac} up to the estimate (\ref{factorization}) with the same notation. In fact, it is enough to prove an analogue of estimate (\ref{itera1}) for the setting under consideration, i.e., we need to show that 
\begin{equation*}
 \sum_{Q\in\mathcal{D}_0}\langle\mathcal{A}_{Q}f_{1}\mathcal{A}_{Q}f_{2}, h_Q\rangle 
 \lesssim  |Q_{0}|\langle f_{1}\rangle_{Q_{0},{\rho}_{1}}\langle f_{2}\rangle_{Q_{0},{\rho}_{2}}\langle h\rangle_{Q_{0},t},
\end{equation*}
where ${\rho}_{1}>r_{1}$, ${\rho}_{2}>r_{2}$ and $\frac{1}{\rho_1}+\frac{1}{\rho_2}< 1$.

In order to use Theorem~\ref{mainthm2:lac}, we first need to decompose functions $f_1$ and $f_2$ into suitable characteristic functions. Consider 
$E_{m}=\{x\in Q_{0}:2^{m}\leq f_{1}(x)\leq 2^{m+1}\}$ and $F_{n}=\{x\in Q_{0}:2^{n}\leq f_{2}(x)\leq 2^{n+1}\}$. Then there exist $m_{0}$, $n_{0}>1$ such that $E_{m}=\emptyset$ for all  $m>m_{0}$ and $F_{n}=\emptyset$ for all  $n>n_{0}$. Denote $f^{m}_{1}=f_{1}\chi_{E_{m}}$ and $f^{n}_{2}=f_{2}\chi_{F_{n}}$. Thus, we use Theorem~\ref{mainthm2:lac} for each pair of characteristic functions $\chi_{E_{m}}$ and $\chi_{F_{n}}$ and obtain the sparse  domination for the functions $f^{m}_{1}$ and $f^{n}_{2}$ as follows
\begin{align*}
\sum_{Q\in\mathcal{D}_0}\langle \mathcal{A}_{Q}f^{m}_{1}\mathcal{A}_{Q}f^{n}_{2}, h_{Q}\rangle &\leq  2^{m+1}2^{n+1}\sum_{Q\in\mathcal{D}_0}\langle \mathcal{A}_{Q}\chi_{E_{m}}\mathcal{A}_{Q}\chi_{F_{n}}, h_{Q}\rangle\\
&\lesssim 2^{m+1}2^{n+1}\sum_{Q\in\mathcal{D}\cap Q_{0}}\langle \mathcal{A}_{Q}\chi_{E_{m}}\mathcal{A}_{Q}\chi_{F_{n}}, h_{Q}\chi_{F_{Q_{0}}}\rangle\\
&\lesssim  2^{m+1}2^{n+1}\sum_{S\in\mathcal{S}_{m,n}}|S|\langle\chi_{E_{m}}\rangle_{S,r_{1}}\langle\chi_{F_{n}}\rangle_{S,r_{2}}\langle h\chi_{F_{Q_{0}}}\rangle_{S,t}
\end{align*}
where $\mathcal{S}_{m,n}$ is the sparse family corresponding to characteristic functions $\chi_{E_{m}}$ and $\chi_{F_{n}}$. 

Next, using the stopping time condition on the function $h$ as given in~(\ref{stph}), we get 
\begin{equation*}
\sum_{S\in\mathcal{S}_{m,n}}|S|\langle\chi_{E_{m}}\rangle_{S,r_{1}}\langle\chi_{F_{n}}\rangle_{S,r_{2}}\langle h\chi_{F_{Q_{0}}}\rangle_{S,t}
 \lesssim  \langle h\rangle_{Q_{0},t}\sum_{S\in\mathcal{S}_{m,n}}|S|\langle\chi_{E_{m}}\rangle_{S,r_{1}}\langle\chi_{F_{n}}\rangle_{S,r_{2}}.
\end{equation*}
Choose $\widetilde{\rho_{1}}>r_{1}$ and $\widetilde{\rho_{2}}>r_{2}$ such that $\frac{1}{\widetilde{\rho}_{1}}+\frac{1}{\widetilde{\rho}_{2}}\leq 1$. When $\frac{1}{\widetilde{\rho}_{1}}+\frac{1}{\widetilde{\rho}_{2}}= 1$, as an easy consequence of the Carleson embedding theorem (see~\cite{Laceysawyer})  we get that  
\begin{align*}
\sum_{S\in\mathcal{S}_{m,n}}|S|\langle\chi_{E_{m}}\rangle_{S,r_{1}}\langle\chi_{F_{n}}\rangle_{S,r_{2}}
&= \sum_{S\in\mathcal{S}_{m,n}}|S|^{\frac{1}{\widetilde{\rho}_{1}}}\langle\chi_{E_{m}}\rangle_{S,r_{1}}|S|^{\frac{1}{\widetilde{\rho}_{2}}}\langle\chi_{F_{n}}\rangle_{S,r_{2}}\\
&\leq  \big(\sum_{S\in\mathcal{S}_{m,n}}|S|\langle \chi_{E_{m}}\rangle^{\widetilde{\rho}_{1}}_{S,r_{1}}\big)^{\frac{1}{\widetilde{\rho}_{1}}}\big(\sum_{S\in\mathcal{S}_{m,n}}|S|\langle \chi_{F_{n}}\rangle^{\widetilde{\rho}_{2}}_{S,r_{2}}\big)^{\frac{1}{\rho}_{2}}\\
&\leq  \langle\chi_{E_{m}}\rangle_{Q_{0},\widetilde{\rho}_{1}}\langle\chi_{F_{n}}\rangle_{Q_{0},\widetilde{\rho}_{2}}|Q_{0}|. 
\end{align*}
Now, for the case when $\frac{1}{\widetilde{\rho}_{1}}+\frac{1}{\widetilde{\rho}_{2}}<1$ we choose $\widetilde{\rho}_3>0$ such that $\frac{1}{\widetilde{\rho}_{1}}+\frac{1}{\widetilde{\rho}_{2}}+\frac{1}{\widetilde{\rho}_{3}}=1$. Then, we have 
\begin{align*}
\sum_{S\in\mathcal{S}_{m,n}}|S|\langle\chi_{E_{m}}\rangle_{S,r_{1}}\langle\chi_{F_{n}}\rangle_{S,r_{2}}
& = \sum_{S\in\mathcal{S}_{m,n}}|S|^{\frac{1}{\widetilde{\rho}_{1}}}\langle\chi_{E_{m}}\rangle_{S,r_{1}}|S|^{\frac{1}{\widetilde{\rho}_{2}}}\langle\chi_{F_{n}}\rangle_{S,r_{2}}|S|^{\frac{1}{\widetilde{\rho}_{3}}}\\
&\lesssim  \langle\chi_{E_{m}}\rangle_{Q_{0},\widetilde{\rho}_{1}}\langle\chi_{F_{n}}\rangle_{Q_{0},\widetilde{\rho}_{2}}|Q_{0}|^{\frac{1}{\widetilde{\rho}_{1}}+\frac{1}{\widetilde{\rho}_{2}}}|Q_{0}|^{\frac{1}{\widetilde{\rho}_{3}}}\\
&=\langle\chi_{E_{m}}\rangle_{Q_{0},\widetilde{\rho}_{1}}\langle\chi_{F_{n}}\rangle_{Q_{0},\widetilde{\rho}_{2}}|Q_{0}|.
\end{align*}
Finally, using~\cite[Lemma 4.6 and Lemma 4.7]{Luz}, we obtain  
\begin{eqnarray*}
|\langle \sup_{Q\in\mathcal{D}_{0}}\mathcal{A}_{Q}f_{1}\mathcal{A}_{Q}f_{2}, h\rangle|\lesssim |Q_{0}|\langle f_{1}\rangle_{Q_{0},{\rho}_{1}}\langle f_{2}\rangle_{Q_{0},{\rho}_{2}}\langle h\rangle_{Q_{0},t},
\end{eqnarray*}
where ${\rho}_{1}>\widetilde{\rho_{1}}$ and ${\rho}_{2}>\widetilde{\rho_{2}}$.

\subsection{Proof of Lemma~\ref{keylac}}
Finally we provide the proof of Lemma~\ref{keylac}. This is the most technical and tedious part of the paper. We will begin by giving the proof in the lacunary case, and after that we will sketch the significantly different parts in case of $\mathcal M_{\operatorname{full}}$.

First note that one can use the same linearisation trick as in the proof of Theorem~\ref{mainthm2:lac}. This would mean that it is enough to prove the following estimate
\begin{equation*}
		\sum_{Q\in\mathcal{D}_{0}}\langle \mathcal{A}_{Q}f_{1}\mathcal{A}_{Q}f_{2},h\chi_{B_{Q}}\rangle\lesssim |Q_{0}|\langle f_{1}\rangle_{Q_{0},r_{1}}\langle f_{2}\rangle_{Q_{0},r_{2}}\langle h\rangle_{Q_{0},t}.
		\end{equation*}
Here we have used the same notation as in Theorem~\ref{mainthm2:lac}. For $i=1,2$, let  
$$
\gamma_{f_{i}}=\{\text{collection of maximal dyadic subcubes} ~ P\subset Q_{0}:\langle f_{i}\rangle_{P,r_{i}}>2C_{0}\langle f_{i}\rangle_{Q_{0},r_{i}}\}.
$$
Applying the Calder\'on-Zygmund decomposition to each $f_i$ at the height $\alpha_i=2C_0\langle f_{i}\rangle_{Q_{0},r_{i}}$, $i=1,2$, we can decompose  
$$
f_{i}=b_{i}+g_{i}, 
$$
where $\| g_{i}\|_{L^{\infty}}\lesssim \langle f_{i}\rangle_{Q_{0},r_{i}}$ and
$$
b_{i}=\sum_{P\in \gamma_{f_{i}}}\big(f_{i}-\langle f_{i}\rangle_{P}\big)\chi_{P}=\sum_{k=-\infty}^{q_0-1}\sum_{P\in B_i(k)}\big(f_{i}-\langle f_{i}\rangle_{P}\big)\chi_{P}=:\sum_{k=-\infty}^{q_0-1}B_{i,k}, 
$$ 
with $l_{Q_0}=2^{q_0}$ and $B_i(k)=\{P\in \gamma_{f_{i}}:l_{P}=2^{k}\}$. Now, 
\begin{align*}
\big|\sum_{Q\in \mathcal{D}_{0}}\langle \mathcal{A}_{Q}f_{1}\mathcal{A}_{Q}f_{2}, h_{Q}\rangle\big|
&\leq  \big|\sum_{Q\in \mathcal{D}_{0}}\langle \mathcal{A}_{Q}g_{1}\mathcal{A}_{Q}g_{2}, h_{Q}\rangle\big|+\big|\sum_{Q\in \mathcal{D}_{0}}\langle \mathcal{A}_{Q}g_{1}\mathcal{A}_{Q}b_{2}, h_{Q}\rangle\big|\\ 
&\quad+\big|\sum_{Q\in \mathcal{D}_{0}}\langle \mathcal{A}_{Q}b_{1}\mathcal{A}_{Q}g_{2}, h_{Q}\rangle\big|+\big|\sum_{Q\in \mathcal{D}_{0}}\langle \mathcal{A}_{Q}b_{1}\mathcal{A}_{Q}b_{2}, h_{Q}\rangle\big| \\
&=:GG+GB+BG+BB. 
\end{align*}
We estimate all the four parts separately. Note that in view of symmetry in $GB$ and $BG$ parts, it is enough to estimate one of them. 

\textbf{Estimate for $GG$ (both functions good) part.} Using the fact that $t> 1$, we have
\begin{align*}
GG \leq  \sum_{Q\in \mathcal{D}_{0}}\| \mathcal{A}_{Q}g_{1}\mathcal{A}_{Q}g_{2}\|_{L^{\infty}}\| h_{Q}\|_{L^{1}}&\lesssim \langle f_{1}\rangle_{Q_{0},r_{1}} \langle f_{2}\rangle_{Q_{0},r_{2}}\sum_{Q\in \mathcal{D}_{0}}\int |h(x)\chi_{B_{Q}}(x)|dx\\
&\lesssim  \langle f_{1}\rangle_{Q_{0},r_{1}} \langle f_{2}\rangle_{Q_{0},r_{2}}\langle h\rangle_{Q_{0}}|Q_{0}|\\ 
&\lesssim  \langle f_{1}\rangle_{Q_{0},r_{1}} \langle f_{2}\rangle_{Q_{0},r_{2}}\langle h\rangle_{Q_{0},t}|Q_{0}|.
\end{align*}

\textbf{Estimate for $BG$ (one function bad and one function good) part.} Arguing with a similar argument as in the proof of Theorem \ref{mainthm2:lac} we note that, for all $Q\in \mathcal{D}_{0}$ and $P\in \gamma_{f_1}$, if $P\cap Q \neq \emptyset$, then $P\subset Q$. Therefore, for any $Q\in \mathcal{D}_{0}$ with $l_Q=2^q$, we have
$$
\langle \mathcal{A}_{Q}b_{1}\mathcal{A}_{Q}g_{2}, h_{Q}\rangle=\sum_{k<q}\langle \mathcal{A}_{Q}B_{1,k}\mathcal{A}_{Q}g_{2}, h_{Q}\rangle=\sum_{k=1}^{\infty}\langle \mathcal{A}_{Q}B_{1,q-k}\mathcal{A}_{Q}g_{2}, h_{Q}\rangle.
$$
Thus,
$$
\big|\sum_{Q\in \mathcal{D}_{0}}\langle \mathcal{A}_{Q}b_{1}\mathcal{A}_{Q}g_{2}, h_{Q}\rangle\big|\lesssim\sum_{k=1}^{\infty}\sum_{Q\in \mathcal{D}_{0}}|\langle \mathcal{A}_{Q}B_{1,q-k}\mathcal{A}_{Q}g_{2}, h_{Q}\rangle|.
$$
Hence, 
\begin{align*}
&BG =\big|\sum_{Q\in\mathcal{D}_{0}}\int b_{1}(x)\mathcal{A}_{Q}^{*}(\mathcal{A}_{Q}g_{2}\cdot h_{Q})(x)dx\big|\\
&\lesssim\big|\sum_{Q\in\mathcal{D}_{0}}\sum^{\infty}_{k=1}\sum_{P\in B_1(q-k)}\int_{P}B_{1,q-k}(x)\mathcal{A}_{Q}^{*}(\mathcal{A}_{Q}g_{2}\cdot h_{Q})(x)dx \big| \\
&\leq  \sum^{\infty}_{k=1}\sum_{Q\in\mathcal{D}_{0}}\sum_{P\in B_1(q-k)}\frac{1}{|P|}\Big|\int_{P}\int_{P}B_{1,q-k}(x)[\mathcal{A}^{*}_{Q}(\mathcal{A}_{Q}g_{2}\cdot h_{Q})(x) -\mathcal{A}^{*}_{Q}(\mathcal{A}_{Q}g_{2}\cdot h_{Q})(x')]dxdx'\Big|.	
\end{align*}
Write $x'=x-y$ for $y\in P_{0}$, where $P_{0}$ is a cube centered at $0$ with side-length $2l_{P}$. We have, by \cite[Lemma 2.3]{Lacey},
\begin{align*}
BG &\lesssim  \sum^{\infty}_{k=1}\sum_{Q\in\mathcal{D}_{0}}\frac{1}{|P_{0}|}\Big|\int_{P_{0}}\int_{Q}B_{1,q-k}(x)[\mathcal{A}^{*}_{Q}(\mathcal{A}_{Q}g_{2}\cdot h_{Q})(x)-\tau_{y}\mathcal{A}^{*}_{Q}(\mathcal{A}_{Q}g_{2}\cdot h_{Q})(x)]dxdy\Big|\\
&\lesssim  \sum^{\infty}_{k=1}\sum_{Q\in\mathcal{D}_{0}}\frac{1}{|P_{0}|}\int_{P_{0}}\Big|\int_{Q}(\mathcal{A}_{Q}B_{1,q-k}-\tau_{-y}\mathcal{A}_{Q}B_{1,q-k})(x)\mathcal{A}_{Q}g_{2}(x)h_{Q}(x)dx\Big|dy\\
&\lesssim  \sum^{\infty}_{k=1}\sum_{Q\in\mathcal{D}_{0}}\frac{1}{|P_{0}|}\int_{P_{0}}\Big(\frac{|y|}{l_{Q}}\Big)^{\eta}|Q|\langle B_{1,q-k}\rangle_{Q,r_{1}}\langle \mathcal{A}_{Q}g_{2}\cdot h_{Q}\rangle_{Q,s_{1}}dy. 
\end{align*}
Further, since $y\in P_{0}$ we have $|y|\lesssim 2^{q-k+1}$. This implies that
$$
BG \lesssim  \sum^{\infty}_{k=1}2^{-\eta k}\sum_{Q\in \mathcal{D}_{0}}|Q|\langle B_{1,q-k}\rangle_{Q,r_{1}}\langle \mathcal{A}_{Q}g_{2}\cdot h_{Q}\rangle_{Q,s_{1}}.
$$
Further, note that 
\begin{align*}
\langle \mathcal{A}_{Q}g_{2}\cdot h_{Q}\rangle_{Q,s_{1}}=\Big(\frac{1}{|Q|}\int_{Q}(\mathcal{A}_{Q}g_{2})^{s_{1}}(x)h^{s_{1}}_{Q}(x)dx\Big)^{\frac{1}{s_{1}}}&\leq \| \mathcal{A}_{Q}g_{2}\|_{L^{\infty}}\langle h_{Q}\rangle_{Q,s_{1}}\\
 &\lesssim \langle f_{2}\rangle_{Q_{0},r_{2}}\langle h_{Q}\rangle_{Q,s_{1}}
\end{align*}
and $\langle B_{1,q-k}\rangle_{Q,r_{1}}\lesssim \langle \chi_{F_{1,q,k}}\rangle_{Q,r_{1}}+\langle f_{1}\rangle_{Q_{0},r_{1}}\langle \chi_{E_{1,q,k}}\rangle_{Q,r_{1}}$, where $F_{1,q,k}$ are disjoint subsets of $F_{1}$ and $E_{1,q,k}$ are disjoint subsets of $Q_{0}$. Putting these estimates together we get
\begin{align*}
BG &\lesssim  \langle f_{2}\rangle_{Q_{0},r_{2}}\sum^{\infty}_{k=1}2^{-\eta k}\sum_{Q\in \mathcal{D}_{0}}|Q|\langle \chi_{F_{1,q,k}}\rangle_{Q,r_{1}}\langle h_{Q}\rangle_{Q,s_{1}}\\
&\quad + \langle f_{1}\rangle_{Q_{0},r_{1}}\langle f_{2}\rangle_{Q_{0},r_{2}}\sum^{\infty}_{k=1}2^{-\eta k}\sum_{Q\in\mathcal{D}_{0}}|Q|\langle \chi_{E_{1,q,k}}\rangle_{Q,r_{1}}\langle h_{Q}\rangle_{Q,s_{1}}=: BG_1 + BG_2.
\end{align*}
We estimate both the terms separately. For the term $BG_1$, note that 
$(\frac{1}{r_{1}},\frac{1}{s_{1}})$ in the interior of $L_{n}$, which implies that $\frac{1}{r_{1}}+\frac{1}{s_{1}}>1$. Choose $\tau>0$ such that $\frac{1}{r_{1}}-\tau+\frac{1}{s_{1}}=1.$ Write $\frac{1}{r_{1}}-\tau=\frac{1}{\dot{r}_{1}}$ and note that $\dot{r}_{1}>r_{1}$. We have, by using that $F_{1,q,k}\subseteq F_1$ and the stopping time condition for the function $f_1$,
\begin{align*}
\langle \chi_{F_{1,q,k}}\rangle_{Q,r_{1}} = \Big(\frac{1}{|Q|}\int_{Q}\chi^{r_{1}}_{F_{1,q,k}}\Big)^{\frac{1}{\dot{r_{1}}}}\Big(\frac{1}{|Q|}\int_{Q}\chi^{r_{1}}_{F_{1,q,k}}\Big)^{\tau}&\leq  \Big(\frac{1}{|Q|}\int_{Q}\chi^{r_{1}}_{F_{1,q,k}}\Big)^{\frac{1}{\dot{r_{1}}}}\Big(\frac{1}{|Q|}\int_{Q}f_{1}^{r_{1}}\Big)^{\tau}\\
& \lesssim  \Big(\frac{1}{|Q|}\int_{Q}\chi^{r_{1}}_{F_{1,q,k}}\Big)^{\frac{1}{\dot{r_{1}}}} \langle f_{1}\rangle^{\tau r_{1}}_{Q_{0},r_{1}}.
\end{align*}
Therefore, as  $t\geq s_1$,
\begin{align}
\label{eq 2}
\notag BG_1 &= \langle f_{2}\rangle_{Q_{0},r_{2}}\sum^{\infty}_{k=1}2^{-\eta k}\sum_{Q\in \mathcal{D}_{0}}|Q|\langle \chi_{F_{1,q,k}}\rangle_{Q,r_{1}}\langle h_{Q}\rangle_{Q,s_{1}}\\
\notag&\lesssim   \langle f_{1}\rangle^{\tau r_{1}}_{Q_{0},r_{1}}\langle f_{2}\rangle_{Q_{0},r_{2}}\sum^{\infty}_{k=1}2^{-\eta k}\sum_{Q\in\mathcal{D}_{0}}|Q|\Big(\frac{1}{|Q|}\int_{Q}\chi_{F_{1,q,k}}\Big)^{\frac{1}{\dot{r}_{1}}}\Big(\frac{1}{|Q|}\int_{Q}h^{s_{1}}_{Q}\Big)^{\frac{1}{s_{1}}}\\
\notag&\lesssim   \langle f_{1}\rangle^{\tau r_{1}}_{Q_{0},r_{1}}\langle f_{2}\rangle_{Q_{0},r_{2}}\sum^{\infty}_{k=1}2^{-\eta k}\Big(\sum_{Q\in\mathcal{D}_{0}}\int_{Q}\chi_{F_{1,q,k}}\Big)^{\frac{1}{\dot{r}_{1}}}\Big(\sum_{Q\in\mathcal{D}_{0}}\int_{Q}h^{s_{1}}\chi_{B_{Q}}\Big)^{\frac{1}{s_{1}}}\\
\notag&\lesssim  \langle f_{1}\rangle^{\tau r_{1}}_{Q_{0},r_{1}}\langle f_{2}\rangle_{Q_{0},r_{2}}\langle f_{1}\rangle^{\frac{r_{1}}{\dot{r}_{1}}}_{Q_{0},r_{1}}|Q_{0}|^{\frac{1}{\dot{r}_{1}}}\langle h\rangle_{Q_{0},s_{1}}|Q_{0}|^{\frac{1}{s_{1}}}\\
&\leq \langle f_{1}\rangle_{Q_{0},r_{1}}\langle f_{2}\rangle_{Q_{0},r_{2}}\langle h\rangle_{Q_{0},t}|Q_{0}|.
\end{align}
Next, the term $BG_2$ may be estimated as follows. Since $\dot{r}_{1}>r_{1}$, we have $\langle \chi_{E_{1,q,k}}\rangle_{Q,r_{1}}\leq \langle \chi_{E_{1,q,k}}\rangle_{Q,\dot{r}_{1}}$. Consider 
 \begin{align}
BG_2 &= \nonumber \langle f_{1}\rangle_{Q_{0},r_{1}}\langle f_{2}\rangle_{Q_{0},r_{2}}\sum^{\infty}_{k=1}2^{-\eta k}\sum_{Q\in\mathcal{D}_{0}}|Q|\langle \chi_{E_{1,q,k}}\rangle_{Q,r_{1}}\langle h_{Q}\rangle_{Q,s_{1}}\\
&\leq \nonumber \langle f_{1}\rangle_{Q_{0},r_{1}}\langle f_{2}\rangle_{Q_{0},r_{2}}\sum^{\infty}_{k=1}2^{-\eta k}\sum_{Q\in\mathcal{D}_{0}}|Q|\langle \chi_{E_{1,q,k}}\rangle_{Q,\dot{r}_{1}}\langle h_{Q}\rangle_{Q,s_{1}}\\
&\leq  \nonumber \langle f_{1}\rangle_{Q_{0},r_{1}}\langle f_{2}\rangle_{Q_{0},r_{2}}\sum^{\infty}_{k=1}2^{-\eta k}\Big(\sum_{Q\in\mathcal{D}_{0}}\int_{Q}\chi_{E_{1,q,k}}\Big)^{\frac{1}{\dot{r}_{1}}}\Big(\sum_{Q\in\mathcal{D}_{0}}\int_{Q}h^{s_{1}}\chi_{B_{Q}}\Big)^{\frac{1}{s_{1}}}\\
&\lesssim \nonumber  \langle f_{1}\rangle_{Q_{0},r_{1}}\langle f_{2}\rangle_{Q_{0},r_{2}}|Q_{0}|^{\frac{1}{\dot{r}_{1}}}\langle h\rangle_{Q_{0},s_{1}}|Q_{0}|^{\frac{1}{s_{1}}}\\
&\leq \label{eq 3} \langle f_{1}\rangle_{Q_{0},r_{1}}\langle f_{2}\rangle_{Q_{0},r_{2}}\langle h\rangle_{Q_{0},t}|Q_{0}|\quad  \text{as}\quad  t\geq s_1.
\end{align}
Estimates \eqref{eq 2} and \eqref{eq 3} yield the desired result for the term $BG$. 
The estimate for the third term $GB$ follows similarly. 

\textbf{Estimate for $BB$ (both functions bad) part.} We have
\begin{align*}
&BB = |\sum_{Q\in\mathcal{D}_{0}}\langle \mathcal{A}_{Q}b_{1}\mathcal{A}_{Q}b_{2},h_{Q}\rangle|=\Big|\sum_{Q\in\mathcal{D}_0}\int b_{1}(x)\mathcal{A}_{Q}^{*}(\mathcal{A}_{Q}b_{2}\cdot h_{Q})(x)dx\Big|\\
&\lesssim\Big|\sum_{Q\in\mathcal{D}_{0}}\sum^{\infty}_{k=1}\sum_{P\in B_1(q-k)}\int_{P}B_{1,q-k}(x)\mathcal{A}_{Q}^{*}(\mathcal{A}_{Q}b_{2}\cdot h_{Q})(x)dx \Big|\\
&\leq \sum^{\infty}_{k=1}\sum_{Q\in\mathcal{D}_{0}}\sum_{P\in B_1(q-k)}\frac{1}{|P|} \Big|\int_{P}\int_{P}B_{1,q-k}(x)[\mathcal{A}^{*}_{Q}(\mathcal{A}_{Q}b_{2}\cdot h_{Q})(x)-\mathcal{A}^{*}_{Q}(\mathcal{A}_{Q}b_{2}\cdot h_{Q})(x')]dx dx'\Big|\\
	&\lesssim 
\sum^{\infty}_{k=1}\sum_{Q\in\mathcal{D}_{0}}\frac{1}{|P_{0}|}\Big|\int_{P_{0}}\int_{Q}B_{1,q-k}(x)[\mathcal{A}^{*}_{Q}(\mathcal{A}_{Q}b_{2}\cdot h_{Q})(x)-\tau_{y}\mathcal{A}^{*}_{Q}(\mathcal{A}_{Q}b_{2}\cdot h_{Q})(x)]dxdy\Big|\\ 
&\lesssim  \sum^{\infty}_{k=1}\sum_{Q\in\mathcal{D}_{0}}\frac{1}{|P_{0}|} \int_{P_{0}}\Big|\int_{Q}(\mathcal{A}_{Q}B_{1,q-k}-\tau_{-y}\mathcal{A}_{Q}B_{1,q-k})(x)\mathcal{A}_{Q}b_{2}(x)h_{Q}(x)dx\Big|dy\\ 
&\lesssim \sum^{\infty}_{k=1}\sum_{Q\in\mathcal{D}_{0}}\frac{1}{|P_{0}|}\int_{P_{0}}\|\mathcal{A}_{Q}B_{1,q-k}-\tau_{-y}\mathcal{A}_{Q}B_{1,q-k})\|_{L^{s_1'}}\| \mathcal{A}_{Q}b_{2}\|_{L^{s'_{2}}}\| h_{Q}\|_{L^{t}}dy\\
&\lesssim  \sum^{\infty}_{k=1}\sum_{Q\in\mathcal{D}_{0}}\frac{1}{|P_{0}|}\int_{P_{0}}\Big(\frac{|y|}{l_{Q}}\Big)^{\eta}|Q|^{1-\frac{1}{s_{1}}}\langle B_{1,q-k}\rangle_{Q,r_{1}}\| \mathcal{A}_{Q}b_{2}\|_{L^{s'_{2}}}\| h_{Q}\|_{L^{t}}dy,
\end{align*}
where, in the last inequality, we have used  \cite[Theorem 2.1]{Lacey} for $\big(\frac{1}{r_1},\frac{1}{s_1'}\big)$ in the interior of $L_n'=\big\{\big(\frac{1}{r},\frac{1}{s}\big): \big(\frac{1}{r},1-\frac{1}{s}\big)\in L_n\big\}$ and H\"{o}lder's inequality with exponents $\frac{1}{t'}=\frac{1}{s'_{1}}+\frac{1}{s'_{2}}$. 
This yields
$$
BB \leq \sum^{\infty}_{k=1}2^{-\eta k}\sum_{Q\in \mathcal{D}_{0}}|Q|^{1-\frac{1}{s_{1}}}\langle B_{1,q-k}\rangle_{Q,r_{1}}\| \mathcal{A}_{Q}b_{2}\|_{L^{s'_{2}}}\Vert h_{Q}\|_{L^{t}}.
$$
Next, we make use of \cite[Lemma 2.3]{Lacey} to estimate the quantity 
$$
\| \mathcal{A}_{Q}b_{2}\|_{L^{s'_{2}}}=\sup_{\|\psi\|_{L^{s_2}}=1}| \langle \mathcal{A}_{Q}b_{2}, \psi\rangle|.
$$ 
Indeed, 
\begin{align*}
	| \langle \mathcal{A}_{Q}b_{2}, \psi\rangle|&=\Big| \int \mathcal{A}_{Q}b_{2}(x)\psi (x)dx\Big|\lesssim\Big|\sum^{\infty}_{j=1}\sum_{P\in B_2(q-j)}\int_{P}B_{2,q-j}(x)\mathcal{A}^{*}_{Q}\psi(x)dx\Big|\\
	&=\Big| \sum^{\infty}_{j=1}\sum_{P\in B_2(q-j)}\frac{1}{|P|}\int_{P}\int_{P}B_{2,q-j}(x)[\mathcal{A}^{*}_{Q}\psi(x)-\mathcal{A}^{*}_{Q}\psi(x')]dxdx'\Big|\\
	&\lesssim \sum^{\infty}_{j=1}\frac{1}{|P_{0}|}\int_{P_{0}}\Big| \int_{Q}B_{2,q-j}(x)[\mathcal{A}^{*}_{Q}\psi(x)-\tau_{y}\mathcal{A}^{*}_{Q}\psi(x)]dx\Big| dy\\
	&\lesssim \sum^{\infty}_{j=1}\frac{1}{|P_{0}|}\int_{P_{0}}|Q|^{1-\frac{1}{s_{2}}}\Big(\frac{|y|}{l_{Q}}\Big)^{\eta}\langle B_{2,q-j}\rangle_{Q,r_{2}}\| \psi\|_{L^{s_{2}}}dy.
	\end{align*}
Thus we obtain the following estimate. 
\begin{eqnarray}
BB & \lesssim & \label{b1b2}\sum^{\infty}_{k,j=1}2^{-\eta (k+j)}\sum_{Q\in\mathcal{D}_{0}}\frac{|Q|^{2}}{|Q|^{\frac{1}{s_{1}}+\frac{1}{s_{2}}}}\langle B_{1,q-k}\rangle_{Q,r_{1}}\langle B_{2,q-j}\rangle_{Q,r_{2}}\Big(\int_{Q}h^{t}_{Q}\Big)^{\frac{1}{t}},	
\end{eqnarray}
where we know that 
\begin{equation}
\label{B1}
\langle B_{1,q-k}\rangle_{Q,r_{1}}\lesssim \langle \chi_{F_{1,q,k}}\rangle_{Q,r_{1}}+\langle f_{1}\rangle_{Q_{0},r_{1}}\langle \chi_{E_{1,q,k}}\rangle_{Q,r_{1}} 
\end{equation}
and 
\begin{equation}
\label{B2}
\langle B_{2,q-j}\rangle_{Q,r_{2}}\lesssim \langle \chi_{F_{2,q,j}}\rangle_{Q,r_{2}}+\langle f_{2}\rangle_{Q_{0},r_{2}}\langle \chi_{E_{2,q,j}}\rangle_{Q,r_{2}}.
\end{equation}
Here, $E_{1,q,k}$, $E_{2,q,j}$ are disjoint subsets of $Q_{0}$ and $F_{1,q,k}$, $F_{2,q,j}$ are disjoint subsets of $F_1$, $F_2$, respectively. 

Substituting \eqref{B1} and \eqref{B2} into \eqref{b1b2}, we get the following four terms, which will be estimated separately.  
\begin{align*}
BB & \lesssim  \label{b12}\sum^{\infty}_{k,j=1}2^{-\eta (k+j)}\sum_{Q\in\mathcal{D}_{0}}\frac{|Q|^{2}}{|Q|^{\frac{1}{s_{1}}+\frac{1}{s_{2}}}}\langle \chi_{F_{1,q,k}}\rangle_{Q,r_{1}} \langle \chi_{F_{2,q,j}}\rangle_{Q,r_{2}}\Big(\int_{Q}h^{t}_{Q}\Big)^{\frac{1}{t}}\\
&\quad + \sum^{\infty}_{k,j=1}2^{-\eta (k+j)}\sum_{Q\in\mathcal{D}_{0}}\frac{|Q|^{2}}{|Q|^{\frac{1}{s_{1}}+\frac{1}{s_{2}}}}\langle \chi_{F_{1,q,k}}\rangle_{Q,r_{1}} \langle f_{2}\rangle_{Q_{0},r_{2}}\langle \chi_{E_{2,q,j}}\rangle_{Q,r_{2}}\Big(\int_{Q}h^{t}_{Q}\Big)^{\frac{1}{t}}\\
&\quad +\sum^{\infty}_{k,j=1}2^{-\eta (k+j)}\sum_{Q\in\mathcal{D}_{0}}\frac{|Q|^{2}}{|Q|^{\frac{1}{s_{1}}+\frac{1}{s_{2}}}}\langle f_{1}\rangle_{Q_{0},r_{1}}\langle \chi_{E_{1,q,k}}\rangle_{Q,r_{1}} \langle \chi_{F_{2,q,j}}\rangle_{Q,r_{2}}\Big(\int_{Q}h^{t}_{Q}\Big)^{\frac{1}{t}}\\
&\quad + \sum^{\infty}_{k,j=1}2^{-\eta (k+j)} \sum_{Q\in\mathcal{D}_{0}}\frac{|Q|^{2}}{|Q|^{\frac{1}{s_{1}}+\frac{1}{s_{2}}}}\langle f_{1}\rangle_{Q_{0},r_{1}}\langle \chi_{E_{1,q,k}}\rangle_{Q,r_{1}} \langle f_{2}\rangle_{Q_{0},r_{2}}\langle \chi_{E_{2,q,j}}\rangle_{Q,r_{2}}\Big(\int_{Q}h^{t}_{Q}\Big)^{\frac{1}{t}}\\
&=: BB_1+BB_2+BB_3+BB_4. 
\end{align*}

{\bf Estimate for the first term $BB_1$.} At this point, one has to deal with the cases $\frac{1}{r_{1}}+\frac{1}{r_{2}}>1$ and $\frac{1}{r_{1}}+\frac{1}{r_{2}}\leq 1$ separately. Let us start with the case $\frac{1}{r_{1}}+\frac{1}{r_{2}}>1$.
Choose positive numbers $\tau_1$ and $\tau_2$ such that $\frac{1}{r_{1}}+\frac{1}{r_{2}}=1+\tau_{1}+\tau_{2}$ and denote $\frac{1}{r_{i}}-\tau_{i}=\frac{1}{\dot{r}_{i}}$, $i=1,2$. Note that $\frac{1}{\dot{r}_{1}}+\frac{1}{\dot{r}_{2}}=1$. We have 
\begin{align*}
BB_1  
&\lesssim  \langle h\rangle_{Q_{0},t}\sum^{\infty}_{k,j=1}2^{-\eta (k+j)}\sum_{Q\in\mathcal{D}_{0}}|Q|^{1-\frac{1}{r_{1}}-\frac{1}{r_{2}}}\Big(\int_{Q}\chi_{F_{1,q,k}}\Big)^{\frac{1}{r_{1}}}\Big(\int_{Q}\chi_{F_{2,q,j}}\Big)^{\frac{1}{r_{2}}}\\
&\lesssim  \langle h\rangle_{Q_{0},t}\langle f_{1}\rangle^{\tau_{1} r_{1}}_{Q_{0},r_{1}}\langle f_{2}\rangle^{\tau_{2} r_{2}}_{Q_{0},r_{2}}\langle f_{1}\rangle^{\frac{r_{1}}{\dot{r}_{1}}}_{Q_{0},r_{1}}\langle f_{2}\rangle^{\frac{r_{2}}{\dot{r}_{2}}}_{Q_{0},r_{2}}|Q_{0}|=\langle f_{1}\rangle_{Q_{0},r_{1}}\langle f_{2}\rangle_{Q_{0},r_{2}}\langle h\rangle_{Q_{0},t}|Q_{0}|.
\end{align*}
The case $\frac{1}{r_1}+\frac{1}{r_2}=1$ follows similarly with $\tau_1=\tau_2=0.$ 

Let us turn to the case when $\frac{1}{r_{1}}+\frac{1}{r_{2}}<1$.
 Observe that $\frac{1}{r_{1}}+\frac{1}{r_{2}}+\frac{1}{t}>1$. Now, choose $\tau_{1},\tau_{2}>0$ such that $\frac{1}{r_{1}}+\frac{1}{r_{2}}+\frac{1}{t}=1+\tau_{1}+\tau_{2}$. This implies $\frac{1}{\dot{r}_{1}}+\frac{1}{\dot{r}_{2}}+\frac{1}{t}=1$, where $\frac{1}{\dot{r}_{i}}=\frac{1}{r_{i}}-\tau_{i}$, for $i=1,2$.
\begin{align*}
	&BB_{1}=\sum^{\infty}_{k,j=1}2^{-\eta(k+j)}\sum_{Q\in\mathcal{D}_{0}} |Q|^{1-\frac{1}{t}}\langle \chi_{F_{1,q,k}}\rangle_{Q,r_{1}} \langle \chi_{F_{2,q,j}}\rangle_{Q,r_{2}}\Big(\int_{Q}h^{t}_{Q}\Big)^{\frac{1}{t}}\\
	&\lesssim\langle f_{1}\rangle^{\tau_{1}r_{1}}_{Q_{0},r_{1}}\langle f_{2}\rangle^{\tau_{2}r_{2}}_{Q_{0},r_{2}}\sum^{\infty}_{k,j=1}2^{-\eta(k+j)}\sum_{Q\in\mathcal{D}_{0}}\Big(\int_{Q}\chi_{F_{1,q,k}}\Big)^{\frac{1}{\dot{r}_{1}}}\Big(\int_{Q}\chi_{F_{2,q,j}}\Big)^{\frac{1}{\dot{r}_{2}}}\Big(\int_{Q}h^{t}_{Q}\Big)^{\frac{1}{t}}\\
	&\lesssim\langle f_{1}\rangle_{Q_{0},r_{1}}\langle f_{2}\rangle_{Q_{0},r_{2}}\langle h\rangle_{Q_{0},t}|Q_{0}|.
	\end{align*}
In the last inequality we have used the H\"older's inequality with respect to $\dot{r_{1}},\dot{r_{2}}$ and  $t$.  

The latter case is analogous for the remaining three terms $BB_2$, $BB_3$ and $BB_4$, hence we will focus on the estimates only for the case when $\frac{1}{r_{1}}+\frac{1}{r_{2}}>1$.  

{\bf Estimate for the second and third terms $BB_2$ and $BB_3$.} 
The estimates for $BB_2$ and $BB_3$ may be obtained in a similar fashion. We provide here the argument for the term $BB_3$. 

Since $\frac{1}{r_{1}}+\frac{1}{r_{2}}>1$, we can choose a positive number $\tau$ such that $\frac{1}{r_{1}}-\tau +\frac{1}{r_{2}}=1$. Denote $\frac{1}{r_{1}}-\tau =\frac{1}{\dot r_{1}}$ and note that  $\frac{1}{\dot{r}_{1}}+\frac{1}{r_{2}}=1$ and $r_{1}<\dot{r}_{1}$. Then we have,
\begin{align*}
BB_3
&\lesssim  \langle f_{1}\rangle_{Q_{0},r_{1}}\langle h\rangle_{Q_{0},t}\sum^{\infty}_{k,j=1}2^{-\eta (k+j)}\sum_{Q\in\mathcal{D}_{0}}|Q|^{1-\frac{1}{r_{1}}-\frac{1}{r_{2}}+\tau}\Big(\int_{Q}\chi_{E_{1,q,k}}\Big)^{\frac{1}{\dot{r}_{1}}}\Big(\int_{Q}\chi_{F_{2,q,j}}\Big)^{\frac{1}{r_{2}}}\\
&\leq  \langle f_{1}\rangle_{Q_{0},r_{1}}\langle h\rangle_{Q_{0},t}\sum^{\infty}_{k,j=1}2^{-\eta (k+j)}\Big(\sum_{Q\in\mathcal{D}_{0}}\int_{Q}\chi_{E_{1,q,k}}\Big)^{\frac{1}{\dot{r}_{1}}}\Big(\sum_{Q\in\mathcal{D}_{0}}\int_{Q}\chi_{F_{2,q,j}}\Big)^{\frac{1}{r_{2}}}\\
&\lesssim  \langle f_{1}\rangle_{Q_{0},r_{1}}\langle h\rangle_{Q_{0},t}|Q_{0}|^{\frac{1}{\dot{r}_{1}}}\langle f_{2}\rangle_{Q_{0},r_{2}}|Q_{0}|^{\frac{1}{r_{2}}}=  \langle f_{1}\rangle_{Q_{0},r_{1}}\langle f_{2}\rangle_{Q_{0},r_{2}}\langle h\rangle_{Q_{0},t}|Q_{0}|.
\end{align*} 

{\bf Estimate for the fourth term $BB_4$.} Choose $\tau_1$ and $\tau_2$ as in the case $BB_1$.  
Consider  
\begin{align*}
BB_4 &\lesssim  \langle f_{1}\rangle_{Q_{0},r_{1}}\langle f_{2}\rangle_{Q_{0},r_{2}}\langle h\rangle_{Q_{0},t} \sum^{\infty}_{k,j=1}2^{-\eta (k+j)}\sum_{Q\in\mathcal{D}_{0}}|Q|^{1-\frac{1}{r_{1}}-\frac{1}{r_{2}}}\Big(\int_{Q}\chi_{E_{1,q,k}}\Big)^{\frac{1}{r_{1}}}\Big(\int_{Q}\chi_{E_{2,q,j}}\Big)^{\frac{1}{r_{2}}}\\
&\leq  \langle f_{1}\rangle_{Q_{0},r_{1}}\langle f_{2}\rangle_{Q_{0},r_{2}}\langle h\rangle_{Q_{0},t}\sum^{\infty}_{k,j=1}2^{-\eta (k+j)}\Big(\sum_{Q\in\mathcal{D}_{0}}\int_{Q}\chi_{E_{1,q,k}}\Big)^{\frac{1}{\dot{r}_{1}}}\Big(\sum_{Q\in\mathcal{D}_{0}}\int_{Q}\chi_{E_{2,q,j}}\Big)^{\frac{1}{\dot{r}_{2}}}\\
&\lesssim  \langle f_{1}\rangle_{Q_{0},r_{1}}\langle f_{2}\rangle_{Q_{0},r_{2}}\langle h\rangle_{Q_{0},t}|Q_{0}|.
\end{align*}

This completes the proof of Lemma~\ref{keylac} for the lacunary bilinear spherical maximal operator. 

For the case of the full bilinear spherical maximal operator, recall the notation introduced in the proof of Theorem~\ref{mainthm2:lac}. We use the Calder\'{o}n-Zygmund decomposition to write $f_i=g_i+b_i$, $~i=1,2$ to get the following 
\begin{align*}
|\sum_{Q\in \mathcal{D}_0}\langle \widetilde{\mathcal{M}}_{Q}(f_{1},f_{2}), h_{Q}\rangle|
&\leq  |\sum_{Q\in \mathcal{D}_0}\langle \widetilde{\mathcal{M}}_{Q}(g_{1},g_{2}), h_{Q}\rangle|+|\sum_{Q\in \mathcal{D}_0}\langle \widetilde{\mathcal{M}}_{Q}(g_{1},b_{2}), h_{Q}\rangle|\\ 
	 	&+|\sum_{Q\in \mathcal{D}_0}\langle \widetilde{\mathcal{M}}_{Q}(b_{1},g_{2}), h_{Q}\rangle|+|\sum_{Q\in \mathcal{D}_0}\langle \widetilde{\mathcal{M}}_{Q}(b_{1},b_{2}), h_{Q}\rangle|\\
&=: GG+GB+BG+BB. 
\end{align*}
\noindent
\textbf{Estimate for $GG$ (both functions good).} In this case we have
 \begin{align*}
\sum_{Q\in \mathcal{D}_0}|\langle \widetilde{\mathcal{M}}_{Q}(g_{1},g_{2}), h_{Q}\rangle| 
&\leq  \sum_{Q\in\mathcal{D}_0}\| \widetilde{M}_{Q}g_{1}\|_{L^{\infty}}\| \widetilde{M}_{Q}g_{2}\Vert_{L^{\infty}}\| h_{Q}\|_{L^{1}}\\
&\lesssim  \langle f_{1}\rangle_{Q_{0},r_{1}}\langle f_{2}\rangle_{Q_{0},r_{2}}\sum_{Q\in\mathcal{D}_0}\int |h(x)|\chi_{B_{Q}}(x)dx\\
&\lesssim  \langle f_{1}\rangle_{Q_{0},r_{1}}\langle f_{2}\rangle_{Q_{0},r_{2}}\langle h\rangle_{Q_{0}}|Q_{0}|.
\end{align*}
\noindent
\textbf{Estimate for $BG$ (one function good and one function bad).} We have
\begin{align*}
&|\sum_{Q\in \mathcal{D}_0}\langle \widetilde{\mathcal{M}}_{Q}(b_{1},g_{2}), h_{Q}\rangle|
\leq \sum_{Q\in\mathcal{D}_0}\Big|\int b_{1}(x)\mathcal{A}_{t_Q}^{*}(\mathcal{A}_{t_Q}g_{2}\cdot h_{Q})(x)dx\Big|\\
&\lesssim  \sum_{k\geq1}\sum_{Q\in\mathcal{D}_0}\frac{1}{|P_{0}|}\int_{P_{0}}\Big|\int_{Q}B_{1,q-k}(x)[\mathcal{A}_{t_{Q}}^{*}(\mathcal{A}_{t_{Q}}g_{2}\cdot h_{Q})(x)-\tau_{y}\mathcal{A}_{t_{Q}}^{*}(\mathcal{A}_{t_{Q}}g_{2}\cdot h_{Q})(x)]dxdy\Big|\\
&\lesssim  \sum_{k\geq1}\sum_{Q\in\mathcal{D}_0}\frac{1}{|P_{0}|}\int_{P_{0}}\Big(\frac{|y|}{l_{Q}}\Big)^{\eta}|Q|\langle B_{1,q-k}\rangle_{Q,r_{1}}\langle \mathcal{A}_{t_{Q}}g_{2}\cdot h_{Q}\rangle_{Q,s_{1}}dy\\
&\lesssim  \sum_{k\geq1}2^{-k\eta}\sum_{Q\in\mathcal{D}_0}|Q|\langle B_{1,q-k}\rangle_{Q,r_{1}}\langle \mathcal{A}_{t_{Q}}g_{2}\cdot h_{Q}\rangle_{Q,s_{1}},
\end{align*}
where we have used \cite[Theorem 3.2]{Lacey} in the second to last inequality. Next, observe that 
$$\langle \mathcal{A}_{t_{Q}}g_{2}\cdot h_{Q}\rangle_{Q,s_{1}}\lesssim \langle f_{2}\rangle_{Q_{0},r_{2}}\langle h_{Q}\rangle_{Q,s_{1}}.$$
This point onward, we can follow the proof in the case of bilinear lacunary spherical maximal function and get the desired estimates. We skip the details. 

This completes the proof of Lemma \ref{keylac}. 
 
\section{Necessary conditions for the sparse domination} 
\label{neces}

In this section we discuss several relations involving the exponents $r_1, r_2, s_1, s_2$ and $t$ and show that they are necessary conditions for the validity of the sparse domination of the bilinear (both lacunary and full) spherical maximal functions. We make use of examples in the spirit of Knapp and Stein \cite{Stein}, the approach in the (linear) sparse domination setting is developed in \cite{Lacey}. 
\subsection{Sparse form for $\mathcal M_{\operatorname{lac}}$} 
Let $f_{1}=f_{2}=\chi_{||x|-1|<\delta}$ and $h=\chi_{|x|\leq c\delta}$ for some $0<\delta<1/4$ and $c\in (0,\frac{1}{2})$. Then we get that 
$\mathcal{A}_{1}f_{1}(x)\geq c h(x)$. Therefore, the sparse domination for the operator $\mathcal M_{\operatorname{lac}}$ implies that 
$$\delta^{n}\lesssim \sum_{Q\in\mathcal{S}}|Q|\langle f_{1}\rangle_{Q,r_{1}}\langle f_{2}\rangle_{Q,r_{2}}\langle h\rangle_{Q,t},$$
where $\mathcal{S}$ is a sparse collection. 

Observe that in the estimate above, in order to make non-trivial contribution to the term on the right hand side, the cube $Q\in \mathcal{S}$ must necessarily  intersect with the supports of $f_{1},f_{2}$ and $h$. Therefore, we may assume that each $Q$ contains the set $\{x: |x|<2\}$. Further, the contribution from a cube decreases as its size increases, therefore it suffices to assume that $\mathcal{S}$ consists of one such cube $Q$. We have the estimate 
$$
\delta^{n}\lesssim \|f_{1}\|_{L^{r_{1}}} \| f_{2}\|_{L^{r_{2}}}\| h\|_{L^{t}}\lesssim \delta^{\frac{1}{r_{1}}+\frac{1}{r_{2}}+\frac{n}{t}}.
$$
Since $\delta>0$ can be chosen arbitrarily small, we get that 
$$
{\frac{1}{r_{1}}+\frac{1}{r_{2}}+\frac{n}{t}}\leq n.
$$ 
Note that the estimate above forces the condition $t>1$. Substituting the value of $t$ in terms of $s_1$ and $s_2$, we get the following necessary condition
\begin{equation}
\label{condl1}
\frac{1}{r_{1}}+\frac{n}{s_{1}}+\frac{1}{r_{2}}+\frac{n}{s_{2}} \leq 2n.
\end{equation}
In a similar fashion, one can show that if $f_{1}=f_{2}=\chi_{|x|<\delta}$ and $h=\chi_{||x|-1|<c\delta}$ for some $0<\delta<1/4$ and $0<c<\frac{1}{2}$, then we get that $\mathcal{A}_{1}f_{1}(x)\geq c\delta^{n-1}h(x)$. This gives us another necessary condition, namely $\frac{n}{r_{1}}+\frac{n}{r_{2}}+\frac{1}{t}\leq 2n-1$. This would mean that 
\begin{equation}
\label{condl2}
\frac{n}{r_{1}}+\frac{1}{s_{1}}+\frac{n}{r_{2}}+\frac{1}{s_{2}}\leq 2n.
\end{equation}
The conditions ~\eqref{condl1} and \eqref{condl2} imply that 
both of $(\frac{1}{r_{i}},\frac{1}{s_{i}})$, $i=1,2$, 
cannot lie outside of the triangle $L_{n}$.

Next, take $f_{1}=\chi_{|x|<\delta}$, $f_{2}=\chi_{|x|<2}$ (also interchanging $f_1$ and $f_2$) and $h=\chi_{||x|-1|<c\delta}$ for some $0<\delta<1$ and $0<c<\frac{1}{2}$ and observe that $$\delta^{n-1}\delta\lesssim \delta^{\frac{n}{r_{1}}}\delta^{\frac{1}{t}}.$$
This yields that $\frac{n}{r_{i}}+\frac{1}{t}\leq n$, $i=1,2$.
Similarly, by taking $f_{1}=\chi_{||x|-1|<\delta}$, $f_{2}=\chi_{|x|<2}$, $h=\chi_{|x|<c\delta}$ for some  $0<\delta<1/4$ and $0<c<\frac{1}{2}$ and interchanging the roles of $f_1$ and $f_2$ we get that $
\frac{1}{r_{i}}+\frac{n}{t}\leq n$, $i=1,2$.

Putting the above two conditions together we get the following condition. 
\begin{equation}
\label{condl3}\max\Big\{\frac{n}{r_{i}}+\frac{1}{t},\frac{1}{r_{i}}+\frac{n}{t}\Big\} \leq n, \quad i=1,2.
\end{equation}
The conditions~(\ref{condl1}), (\ref{condl2}) and (\ref{condl3}) must necessarily be satisfied for the sparse domination of the operator $\mathcal M_{\operatorname{lac}}$ to hold. However, due to the techniques of our proof we are getting an additional condition on the exponents, namely, 
\begin{equation}
\label{extra}
\frac{1}{r_1}+\frac{1}{r_2}<1.
\end{equation}

\subsection{Sparse form for $\mathcal M_{\operatorname{full}}$}
Consider $f_{1}=|x|^{1-n}(\log \frac{1}{|x|})^{-1}\chi_{|x|<\frac{3}{4}}$ and $f_{2}=\chi_{|x|<1}$ and note that $f_{1}\in L^{r_{1}}(\R^n)$ for $1<r_{1}\leq \frac{n}{n-1}$. It is easy to verify that $\mathcal{M}_{\operatorname{full}}(f_{1},f_{2})$ is infinite on a set of positive measure. This gives us the condition that  $\frac{1}{r_{1}}< \frac{n-1}{n}$.  Using the symmetry between $f_{1}$ and $f_{2},$ we also have that $\frac{1}{r_{2}}<\frac{n-1}{n}$. Next, we observe that both of $(\frac{1}{r_{i}},\frac{1}{s_{i}})$, $i=1,2$ cannot lie above the line segment $P_{1}P_{4}$ in $F_{n}$, see Figure \ref{LnFn}. This can be proved by considering the functions $f_{1}=f_{2}=\chi_{||x|-1|<\delta}$ and $h=\chi_{|x|\leq c\delta}$ for some $0<\delta<1/4$ and $c\in (0,\frac{1}{2})$. This is same as in the case of lacunary maximal function. We omit the details. 

Consider $f_{1}=f_{2}=\chi_{R_{1}}$ and $h=\chi_{R_{2}}$, where $R_{1}=[-C\sqrt\delta, C\sqrt\delta]^{n-1}\times [-C\delta, C\delta]$ and $R_{2}=[-\sqrt\delta,\sqrt\delta]^{n-1}\times [\frac{4}{3},\frac{5}{3}]$. This yields 
\begin{equation*}
\langle \widetilde{\mathcal{M}}(f_{1},f_{2}),h\rangle \gtrsim \delta^{\frac{3(n-1)}{2}}.
\end{equation*}
The sparse domination of $\langle \widetilde{\mathcal{M}}(f_{1},f_{2}),h\rangle$ yields  
\begin{equation*}
\delta^{\frac{3(n-1)}{2}}\leq \delta^{\frac{n+1}{2r_{1}}}\delta^{\frac{n+1}{2r_{2}}}\delta^{\frac{n-1}{2t}}.
\end{equation*}
This gives us the condition 
\begin{equation}\label{condf1}
\frac{n+1}{r_{1}}+\frac{n-1}{s_{1}}+\frac{n+1}{r_{2}}+\frac{n-1}{s_{2}}\leq 4(n-1).
\end{equation}
Therefore, both of $(\frac{1}{r_{i}},\frac{1}{s_{i}})$, $i=1,2$, cannot lie above the line segment $P_{3}P_{4}$ in Figure \ref{LnFn}.


Also, the conditions
$\frac{1}{r_{i}}+\frac{n}{t}\leq n$, $i=1,2$, must be satisfied for the sparse domination of the full maximal function as they hold for the lacunary maximal function. Further, by considering $f_{1}=\chi_{R_{1}}$, $f_{2}=\chi_{B((0,0,....,\frac{4}{3}),2)}$ and  $h=\chi_{R_{2}}$, we obtain that
\begin{equation*}
\delta^{n-1}\lesssim \langle \mathcal{M}_{\operatorname{full}}(f_{1},f_{2}),h\rangle\lesssim \delta^{\frac{n+1}{2r_{1}}}\delta^{\frac{n-1}{2t}}.
\end{equation*} 
Therefore, we get that
\begin{equation*}
\frac{n+1}{r_{1}}+\frac{n-1}{t}\leq 2(n-1).
\end{equation*}
Interchanging the roles of $f_{1}$ and $f_{2}$, we also have that 
\begin{equation*}
\frac{n+1}{r_{2}}+\frac{n-1}{t}\leq 2(n-1).
\end{equation*}
These are necessary conditions on various parameters in order the sparse domination to hold for the full bilinear spherical maximal function. Restriction \eqref{extra} also arises in this case. 

\begin{remark}
The necessary condition \eqref{extra} arises because we need $\frac{1}{\rho_1}+\frac{1}{\rho_2}<1$ for the sparse domination in Theorem \ref{mainthm1:lac}. An inspection on the proof of Theorem \ref{mainthm1:lac} shows that this condition is not required for proving the sparse domination when the functions are characteristic functions, but for the general functions. We guess that this could give restricted weak-type weighted results for a better range of exponents.
\end{remark}

\section*{Acknowledgements}

The first author is supported by the Basque Government through the BERC 2018-2021 program, by Spanish Ministry of Economy and Competitiveness MINECO: BCAM Severo Ochoa excellence accreditation SEV-2017-2018 and through project MTM2017-82160-C2-1-P funded by (AEI/FEDER, UE) and acronym ``HAQMEC''. She also acknowledges the RyC project RYC2018-025477-I. The second author acknowledges the financial support from the Science and Engineering Research Board (SERB), Government of India, under the grant MATRICS: MTR/2017/000039/Math. The third author is supported by CSIR (NET), file no. 09/1020 (0094)/2016-EMR-I. 

All the three authors are thankful to Jos\'e Mar\'ia Martell and Bas Nieraeth for their suggestions and remarks. 



\begin{thebibliography}{99}
\bibitem{AP} T. C. Anderson; E. A. Palsson, 
{\it Bounds for discrete multilinear spherical maximal functions in higher dimensions,} arXiv:1910.12458.

\bibitem {Luz} S. Bagchi; S. Hait; L. Roncal; S. Thangavelu, {\it On the maximal function associated to the lacunary spherical means on the Heisenberg group,} arXiv:1812.11926.


\bibitem{BGHH} J. A. Barrionevo; L. Grafakos; D. He; P. Honz\'{i}k; L. Oliveira, {\it Bilinear spherical maximal function,} 
Math. Res. Lett. 25 (2018), no. 5, 1369--1388. 

%

\bibitem{Bourgain} J. Bourgain,
{\it Averages in the plane over convex curves and maximal operators},
J. Anal. Math.
47 (1986), 69--85.

\bibitem{Calderon} A. P. Calder\'on; A. Zygmund, {\it A note on the interpolation of linear operations,} Studia Math. 12 (1951), 194--204. 

\bibitem{Calderon2} C. P. Calder\'{o}n, 
{\it Lacunary spherical means,} 
Illinois J. Math. 23 (1979), no. 3, 476--484.

\bibitem{CM} R. R. Coifman; Y. Meyer, 
{\it Commutateurs d'int\'egrales singuli\`eres et op\'eraterus multilin\'eaires,} Ann. Inst. Fourier (Grenoble) 28 (1978), 177--202.

\bibitem{CW} R. R. Coifman; G. Weiss, 
{\it Book review: Littlewood--Paley and multiplier theory,} 
Bull. Amer. Math. Soc. 84 (1978), no. 2, 242--250.

\bibitem{JM} J. M. Conde-Alonso; G. Rey, 
{\it A pointwise estimate for positive dyadic shifts and some applications,} Math. Ann. 365 (2017), 1111--1135.

\bibitem{CCG} M. Cowling; J. G. Cuerva; H. Gunawan, 
{\it Weighted estimates for fractional maximal functions related to spherical means,} Bull. Austral. Math. Soc. 66 (2002), no. 1, 75--90.

\bibitem{DLP}  W. Dami\'{a}n; A. Lerner; C. P\'{e}rez, 
{\it Sharp weighted bounds for multilinear maximal functions and Calder\'on-Zygmund operators.} J. Fourier Anal. Appl. 21 (2015), no. 1, 161--181.

\bibitem{Dosidis} G. Dosidis, 
{\it Multilinear spherical maximal function,} arXiv:1911.04071.

\bibitem{DMO} J. Duoandikoetxea; A. Moyua; O. Oruetxebarria, {\it The spherical maximal operator on radial functions,} J. Math. Anal. Appl. 387 (2012), no. 2, 655--666.

\bibitem{DV} J. Duoandikoetxea; L. Vega, 
{\it Spherical means and weighted inequalities,} J. London Math. Soc. (2) 53 (1996), no. 2, 343--353.

\bibitem{GGIP} D. A. Geba; A. Greenleaf; A. Iosevich; E. A. Palsson; E. Sawyer, 
{\it Restricted convolution inequalities, multilinear operators and applications,} Math. Res. Lett. 20 (2013), no. 4, 675--694.

\bibitem {GHH} L. Grafakos; D. He; P. Honz\'{i}k, {\it Maximal operators associated with bilinear multipliers of limited decay,} J. Anal. Math. (to appear) and available at arXiv:1804.08537.


\bibitem{GT}  L. Grafakos; R. H. Torres, 
{\it Multilinear Calder\'{o}n-Zygmund theory.} Adv. Math. 165 (2002), no. 1, 124--164. 

\bibitem{JL} E. Jeong; S. Lee, 
{\it Maximal estimates for the bilinear spherical averages and the bilinear Bochner-Riesz operators,} arXiv:1903.07980.

\bibitem{JSK} K. Jotsaroop; S. Shrivastava; K. Shuin, 
{\it Weighted estimates for bilinear Bochner-Riesz means at the critical index,} preprint.

\bibitem{Lacey} M. T. Lacey, 
{\it Sparse bounds for spherical maximal functions}, J. Anal. Math. 139 (2019), no. 2, 613--635.

\bibitem{Laceysawyer} M. T. Lacey; E. Sawyer; I. Uriarte-Tuero, 
{\it Two weight inequalities for discrete positive operators,} arXiv:0911.3437v4.

\bibitem{LT1} M. T. Lacey; C. Thiele,
{\it $L^p$ estimates on the bilinear Hilbert transform for $2<p<\infty$,} 
 Ann. of Math.(2) 146 (1997), no. 3, 693--724.

\bibitem{LT2} M. T. Lacey; C. Thiele,
{\it On Calder\'on's conjecture,}
Ann. of Math.(2) 149 (1999), no. 2, 475--496.

\bibitem{LN}
A. K. Lerner; F. Nazarov,
{\it Intutive dyadic calculus: the basics,} 
Expo. Math. 37 (2019), no. 3, 225--265.

\bibitem {Lerner} A. K. Lerner; S. Ombrosi; C. P\'erez; R. H. Torres; R. Trujillo-Gonz\'alez, 
{\it New maximal functions and multiple weights for the multilinear Calder\'on--Zygmund theory,} Adv. Math. 220 (2009) 1222--1264.

\bibitem{LMO1} K. Li; J. M. Martell; H. Martikainen; S. Ombrosi; E. Vuorinen, 
{\it End-point estimates, extrapolation for multilinear Muckenhoupt classes, and applications,}
arXiv:1902.04951v1.

\bibitem{LMO} K. Li; J. M. Martell; S. Ombrosi, 
{\it Extrapolation for multilinear Muckenhoupt classes and applications to the bilinear Hilbert transform,}
arXiv:1802.03338.

\bibitem{Manna} R. Manna, 
{\it Weighted inequalities for spherical maximal operator,} Proc. Japan Acad. Ser. A Math. Sci. 91 (2015), no. 9, 135--140.

\bibitem{MSS}
G. Mockenhaupt; A. Seeger; C. Sogge,
{\it Wave front sets, local smoothing and Bourgain's circular maximal theorem},
Ann. of Math. 136 (1992), 207--218.

\bibitem {Nie} B. Nieraeth, 
{\it Quantitative estimates and extrapolation for multilinear weight classes,} Math. Ann. 375 (2019), no. 1-2, 453--507. 

\bibitem {XQ} X. Shi; Q. Sun, 
{\it Weighted norm inequalities for Bochner-Riesz operators and singular integral operators.}  Proc. Amer. Math. Soc. 116 (1992), no. 3, 665--673.


\bibitem{Stein} E. M. Stein, {\it Maximal functions. I. Spherical means,} Proc. Nat. Acad. Sci. U.S.A. 73 (1976),
no. 7, 2174--2175.

\end{thebibliography}
\end{document}